\documentclass[11pt]{article}
\usepackage{amsthm,amsmath,amssymb,amsfonts,amsopn,color,graphics,graphicx}
\usepackage{enumerate}
\usepackage{color}

\newcommand{\RN}{\mathbb{R}^N}

\newcommand{\gam}{\gamma}
\newcommand{\vphi}{\varphi}
\newcommand{\into}{\int_{\Omega}}
\newcommand{\intr}{\int_{\mathbb{R}^N}}
\newcommand{\hmu}{\overline{\mu}}
\newcommand{\ep}{\epsilon}

\newtheorem{thm}{Theorem}[section]
\newtheorem{df}{Definition}[section]
\newtheorem{prop}[thm]{Proposition}
\newtheorem{cor}[thm]{Corollary}
\newtheorem{lem}[thm]{Lemma}
\newtheorem{rem}{Remark}[section]

\begin{document}

\title{Stable solutions and finite Morse index solutions of
nonlinear elliptic equations with Hardy potential}

\author{Wonjeong Jeong$^1$ and Youngae Lee$^2$ \\
$^1$ Department of Mathematics, POSTECH, \\
Pohang, Kyungbuk 790-784, Republic of Korea \\
$^2$ Center for Advanced Study in Theoretical Science, National Taiwan University, \\
No.1, Sec. 4, Roosevelt Road, Taipei 106, Taiwan\\
E-mail: $^1$thewonj@postech.ac.kr, $^2$youngaelee0531@gmail.com}
\date{}
\maketitle

\begin{abstract}
We are concerned with Liouville-type results of
stable solutions and finite Morse index solutions
for the following nonlinear elliptic equation with Hardy potential:
\begin{displaymath}
\Delta u+\dfrac{\mu}{|x|^2}u+|x|^l |u|^{p-1}u=0 \qquad \textrm{in}\ \ \Omega,
\end{displaymath}
where $\Omega=\RN$, $\RN\setminus\{0\}$ for $N\geq3$,
$p>1$, $l>-2$ and $\mu<(N-2)^2/4$.
Our results depend crucially on a new critical exponent $p=p_c(l,\mu)$ and
the parameter $\mu$ in Hardy term.
We prove that there exist no nontrivial stable solution
and finite Morse index solution for $1<p<p_c(l,\mu)$.
We also  {observe} a range of the exponent $p$ larger than $p_c(l,\mu)$
satisfying that our equation admits a positive radial stable solution.
\end{abstract}


\maketitle


\section{Introduction}

We consider stable solutions and finite Morse index solutions
of the following nonlinear elliptic equation with Hardy potential:
\begin{equation}\label{H}
\Delta u+\dfrac{\mu}{|x|^2}u+|x|^l |u|^{p-1}u=0 \qquad \textrm{in}\ \ \Omega,
\end{equation}
where $\Omega=\RN$, $\RN\setminus\{0\}$ for $N\geq3$.
Here $p>1$, $l>-2$ and the parameter $\mu$ satisfies the inequality $\mu<\hmu$,
where $\hmu:=(N-2)^2/4$ is the best constant in Hardy's inequality.
The nonlinear elliptic equations with Hardy potential have been studied by many authors
(see \cite{BDT, Du, T} and the references therein).

Our motivation for investigating (\ref{H}) comes from the results of \cite{F1}
in which Farina established Liouville-type theorems for stable solutions and finite Morse index solutions
of \eqref{H} with $\mu=l=0$ and $\Omega=\RN$($N\geq2$),
which is called by the Lane-Emden equation, as follows:
\begin{equation}\label{LE}
-\Delta u=|u|^{p-1}u \qquad \textrm{in}\ \ \RN.
\end{equation}

Recently, there has been some interest in studying stable or finite Morse index solutions
of the following autonomous elliptic equation:
\begin{equation}\label{elliptic eq}
-\Delta u = f(u) \qquad \textrm{in}\ \ \RN.
\end{equation}
Stable radial solutions of (\ref{elliptic eq}) are well-understood
by the works in \cite{CC} and \cite{V}.
It was shown in \cite{CC} that every bounded radial stable solution must be
constant if $N\leq10$ and $f\in C^1$ satisfies a generic
non-degeneracy condition.
However, in case of nonradial solutions, much less is known.
In \cite{DF1, DF2}, the authors presented Liouville-type theorems
for stable solutions and finite Morse index solutions of (\ref{elliptic eq})
with general convex and non-decreasing nonlinearities $f\geq 0$.
In particular, a complete analysis of stable solutions and finite Morse index
solutions (including solutions which are stable outside a compact set)
is provided for two important nonlinearities $f(u)=|u|^{p-1}u$, $p>1$ and
$f(u)=e^u$ in \cite{DF, F1, F2}.
When $f(u)=|u|^{p-1}u$, Farina's results \cite{F1} say that the equation (\ref{LE}) with
\begin{align*}
1<p<p_c:=\left\{ \begin{array}{ll}
+\infty \qquad &\textrm{if}\ \ N\leq 10,\\
\frac{(N-2)^2-4N+8\sqrt{N-1}}{(N-2)(N-10)} \qquad &\textrm{if}\ \ N\geq11
\end{array}\right.
\end{align*}
has no nontrivial stable solution and it admits a positive radial stable solution
if $p\geq p_c$ and $N\geq11$.
It was also proved that (\ref{LE}) has no nontrivial finite Morse index solution
when $1<p<p_c$ and $p\neq(N+2)/(N-2)$.

Furthermore, in a recent paper \cite{DDG}, Dancer, Du and Guo
extended some results in \cite{F1}.
They considered (\ref{H}) with $\mu=0$ as follows:
\begin{equation}\label{DDG}
-\Delta u=|x|^l |u|^{p-1}u \qquad \textrm{in}\ \ \Omega,
\end{equation}
where $p>1$, $l>-2$ and $\Omega$ is a bounded or unbounded domain of $\RN$ for $N\geq 2$.
It was proved that there is no nontrivial stable solution of (\ref{DDG}) in $\RN$ if
\begin{align*}
1<p<p_c(l):=\left\{ \begin{array}{ll}
+\infty \qquad &\textrm{if}\ \ N\leq 10+4l,\\
\frac{(N-2)^2-2(l+2)(N+l)+2(l+2)\sqrt{(N+l)^2-(N-2)^2}}{(N-2)(N-10-4l)} \qquad &\textrm{if}\ \ N>10+4l
\end{array}\right.
\end{align*}
and that for $p\geq p_c(l)$, (\ref{DDG}) admits a positive radial stable solution in $\RN$.
Moreover, for a finite Morse index solution $u$, the authors obtained the behavior of $u$
near the origin when $\Omega$ is a punctured ball $B_R(0)\setminus\{0\}$ and
its behavior near infinity when $\Omega$ is an exterior domain $\RN\setminus B_R(0)$.
In addition, Wang and Ye \cite{WY, WY0} obtained Liouville-type result
for finite Morse index solutions of \eqref{DDG} in $\RN$,
which is a partial extension of results in \cite{DDG}.
On the other hand, Dancer obtained many results related to stable or finite Morse index solutions
of (\ref{H}) (see \cite{D1, D2}).
For other relevant papers, see \cite{BaL, DDF} and the references therein.

In this direction, we expect that further extension and generalization of the above results
can be made.
Hence we consider a more general non-autonomous equation (\ref{H})
in an entire space $\RN$ and a punctured space $\RN\setminus\{0\}$
to extend some of results in \cite{DDG} and \cite{F1} (partially including \cite{WY}).
Throughout this paper, we verify that there is an exponent $p=p_c(l,\mu)$ depending on $N$, $l$ and $\mu$
such that (\ref{H}) has no nontrivial stable solution in $\RN$ if $1<p< p_c(l,\mu)$
and it admits a positive radial stable solution in $\RN$
for certain range of the exponent $p$ larger than $p_c(l,\mu)$.
We also prove that (\ref{H}) has no nontrivial finite Morse index solution in $\RN$ and $\RN \setminus\{0\}$
if $1<p<p_c(l_-,\mu)$ with $l_-=\min\{l,0\}$ and $p\neq(N+2+2l)/(N-2)$.

Our approach to the problem (\ref{H}) is based on ideas of \cite{DDG} and \cite{F1}.
The main point of their arguments consist in obtaining the integral estimate
which is satisfied by stable solutions.
However, because of the presence of Hardy potential in our case,
the difficulty lies in deriving the integral estimate.
To overcome this difficulty,
we use the method that the Hardy term is absorbed into the other parts
by applying Hardy's inequality.
This idea enables us to apply the methods of \cite{DDG} and \cite{F1}.
Another difficulty stems from the fact that the problem (\ref{H}) is singular.
For this reason, we use a more delicate approach to derive improved versions.

At this point, the most recent work \cite{DG1} should be noticed
although it was written after our paper under review.
In \cite{DG1}, Du and Guo studied the behavior of finite Morse index solutions of the equation
\begin{equation}\label{CKN}
-{\rm{div}}(|x|^{\theta}\nabla v)=|x|^{\alpha} |v|^{p-1}v \qquad\textrm{in}\ \ \Omega,
\end{equation}
where $p>1$, $\theta$, $\alpha\in\mathbb{R}$ and $\Omega$ is a bounded or unbounded domain in $\RN$.
Through the transformation $v=|x|^{\sigma}u$ with $\sigma=\sqrt{\hmu}-\sqrt{\hmu-\mu}$,
the equation \eqref{H} can be reduced to \eqref{CKN} with $\theta=-2\sigma$ and $\alpha=l-\sigma(p+1)$, and vice versa.
By using the fact that the stability of solutions to \eqref{H} is unchanged under natural transformations,
the authors treated the equivalent problem \eqref{CKN} to our main equation \eqref{H}
and also applied the methods in \cite{DDG} and \cite{F1} to \eqref{CKN} after introducing a suitable setting.
They found the critical values of the exponent $p$ in \eqref{CKN}
dividing the behavior of finite Morse index solutions of \eqref{CKN}.

The organization of this paper is as follows.
In Section 2, we state our main results and give some remarks.
In Section 3, we show Liouville-type theorem for stable weak solutions in $\RN$,
that is, the non-existence of nontrivial stable weak solutions.
We also find the fact that there is a dividing curve $p=p_c(l,\mu)$
in the $(\mu,p)$-plane with respect to non-existence of stable solutions.
In Section 4, we prove non-existence of nontrivial finite Morse index weak solutions
in $\RN$ and $\RN\setminus\{0\}$ by using their behaviors near the origin and infinity.
In Section 5, we investigate the stability of positive radial solutions,
and then we find the existence of a stable weak solution in $\RN$
for certain range of $p$ and $\mu$.


\section{Main results}

To describe results of this paper more precisely,
we first introduce the following definitions.

\begin{df}\label{def}
We say that $u$ is a weak solution of \eqref{H} in $\Omega$
if $u\in W^{1,2}_{loc}(\Omega)\cap L^{\infty}_{loc}(\Omega)$ and
\begin{equation}\label{def-wk sol}
\into \nabla u\cdot\nabla\phi-\dfrac{\mu}{|x|^2}u\phi-|x|^l|u|^{p-1}u\phi\,dx=0
\qquad \textrm{for all}\ \ \phi\in C^1_c(\Omega).
\end{equation}
\end{df}

We see that a weak solution of (\ref{H}) in $\Omega$ is a classical solution
of (\ref{H}) in $\Omega\setminus\{0\}$
by the standard elliptic regularity theory in \cite{GT}.

\begin{rem}\label{wk sol}\em
The local boundedness in the above definition could be weakened as in \cite[Definition 1.2]{WY}.
In other words, we can also say that $u$ is a weak solution of \eqref{H} in $\Omega$
if $u\in W^{1,2}_{loc}(\Omega)$ satisfies \eqref{def-wk sol} and $|x|^{-2}u+|x|^l|u|^{p-1}u\in L^1_{loc}(\Omega)$.
Then our results, especially Theorems \ref{thm-stable}, \ref{thm-morse2} and Corollary \ref{cor-morse1},
hold for the weak solution $u$ in the sense of \cite[Definition 1.2]{WY}, and the reason will be explained in Remark \ref{rem-sol}.
In this paper, there are some difficulties generated from behavior and regularity of solutions
near origin due to Hardy term.
Thus, for simplicity, we deal with weak solutions introduced in Definition \ref{def}.
\end{rem}

\begin{df}
A solution $u$ of \eqref{H} in $\Omega$ is said to be stable if
\begin{equation}\label{def-stable}
Q_u(\phi):=\into|\nabla \phi|^2-\dfrac{\mu}{|x|^2}\phi^2-p|x|^l |u|^{p-1}\phi^2\,dx \geq 0
\qquad \textrm{for all}\ \ \phi\in C^1_c(\Omega).
\end{equation}
\end{df}

Since $C^1_c(\Omega)$ is dense in $W^{1,2}_0(\Omega)$ and
$|x|^{-2}$, $|x|^l$ are integrable in a neighborhood of zero
(because $l>-2$ and $N\geq3$),
a test function $\phi$ can be taken from the class
of functions in $W^{1,2}_{loc}(\Omega)\cap L^{\infty}(\Omega)$
with a compact support in $\Omega$.

\begin{df}
We say that a solution $u$ of \eqref{H} has finite Morse index $k\geq0$ if
the integer $k$ is the maximal dimension of a subspace $M$ of $C^1_c(\Omega)$
satisfying $Q_u(\phi)<0$ for any $\phi\in M \setminus \{0\}$.
\end{df}

We note that $u$ is stable if and only if it has Morse index $0$.
Furthermore, every finite Morse index solution $u$ of \eqref{H} is
stable outside a compact set $K\subset\Omega$, which means that
$Q_u(\phi)\geq 0$ for all $\phi\in C^1_c(\Omega\setminus K)$.
Indeed, there exist an integer $k\geq 0$ and
a subspace $M_k:={\rm span}\{\vphi_1,\cdots,\vphi_k\}$ in $C^1_c(\Omega)$
such that $Q_u(\vphi)<0$ for all $\vphi\in M_k\setminus\{0\}$.
It implies that $Q_u(\phi)\geq 0$ for any $\phi\in C^1_c(\Omega\setminus K)$,
where $K:=\bigcup^k_{j=1}{\rm supp}(\vphi_j)$.

Before stating our main results, we introduce a new critical exponent $p=p_c(l,\mu)>1$
which is the unique solution of the following equation for $p>1$:
$$N=2+\dfrac{2(l+2)}{(p-1)\mu_+ +\hmu}\left\{\hmu+\dfrac{\hmu}{p-1}
+\sqrt{\hmu(\hmu-\mu_+)\left(1+\dfrac{1}{p-1}\right)}\right\},$$
where $\mu_+=\max\{\mu,0\}$.
It will be obtained in the proof of Theorem \ref{thm-stable}.

Now we give the non-existence result for nontrivial stable weak solutions of \eqref{H} in $\RN$.

\begin{thm}\label{thm-stable}
If $u$ is a stable weak solution of \eqref{H} in $\Omega=\RN$ with $1<p<p_c(l,\mu)$,
then $u\equiv0$.
Moreover, when $\mu\leq0$, $p_c(l,\mu)$ is constant with respect to $\mu$ given by
\begin{align*}
p_c(l,\mu)=p_c(l,0):=\left\{ \begin{array}{ll}
+\infty \qquad &\textrm{if}\ \ N\leq 10+4l,\\
\frac{(N-2)^2-2(l+2)(N+l)+2(l+2)\sqrt{(N+l)^2-(N-2)^2}}{(N-2)(N-10-4l)} \qquad &\textrm{if}\ \ N>10+4l,
\end{array}\right.
\end{align*}
and when $\mu>0$, $p=p_c(l,\mu)$ is a strictly decreasing function in $\mu$ and satisfies
\begin{align}\label{p_c limit}
\lim_{\mu\rightarrow 0^+}p_c(l,\mu)=p_c(l,0)
\qquad \textrm{and} \qquad
\lim_{\mu\rightarrow \hmu^-}p_c(l,\mu)=\dfrac{N+2+2l}{N-2}.
\end{align}
\end{thm}

\begin{rem}\em
When $\mu=l=0$, a non-existence result of Farina \cite[Theorem 1]{F1} asserts that
if $u$ is a stable solution in $\RN$ and $p$ satisfies
\begin{align*}
\left\{ \begin{array}{ll}
1<p<+\infty \qquad &\textrm{if}\ \ 2\leq N\leq 10,\\
1<p<p_c(0,0)=\frac{(N-2)^2-4N+8\sqrt{N-1}}{(N-2)(N-10)} \qquad &\textrm{if}\ \ N\geq11,
\end{array}\right.
\end{align*}
then $u$ is identically zero.
Recently this result was extended into a more general case $\mu=0$ and $l>-2$
in \cite[Theorem 1.2]{DDG}.
In this regard, Theorem \ref{thm-stable} is a natural extension
of results in \cite{DDG} and \cite{F1}.
The proof of Theorem \ref{thm-stable} is based on the methods
introduced in \cite{F1}.
\end{rem}

The exponent $p=p_c(l,\mu)$ plays a crucial role in this paper.
Indeed, the exponent $p=p_c(l,\mu)$ determines
the region of $(\mu,p)$ in which (\ref{H}) does not admit
stable weak solutions.
See the regions (``Unstable" regions) below the curve $p=p_c(l,\mu)$ in Figure 1 and Figure 2.

\begin{figure}[hbt]
\centering
\centerline{\includegraphics[scale = 0.6]{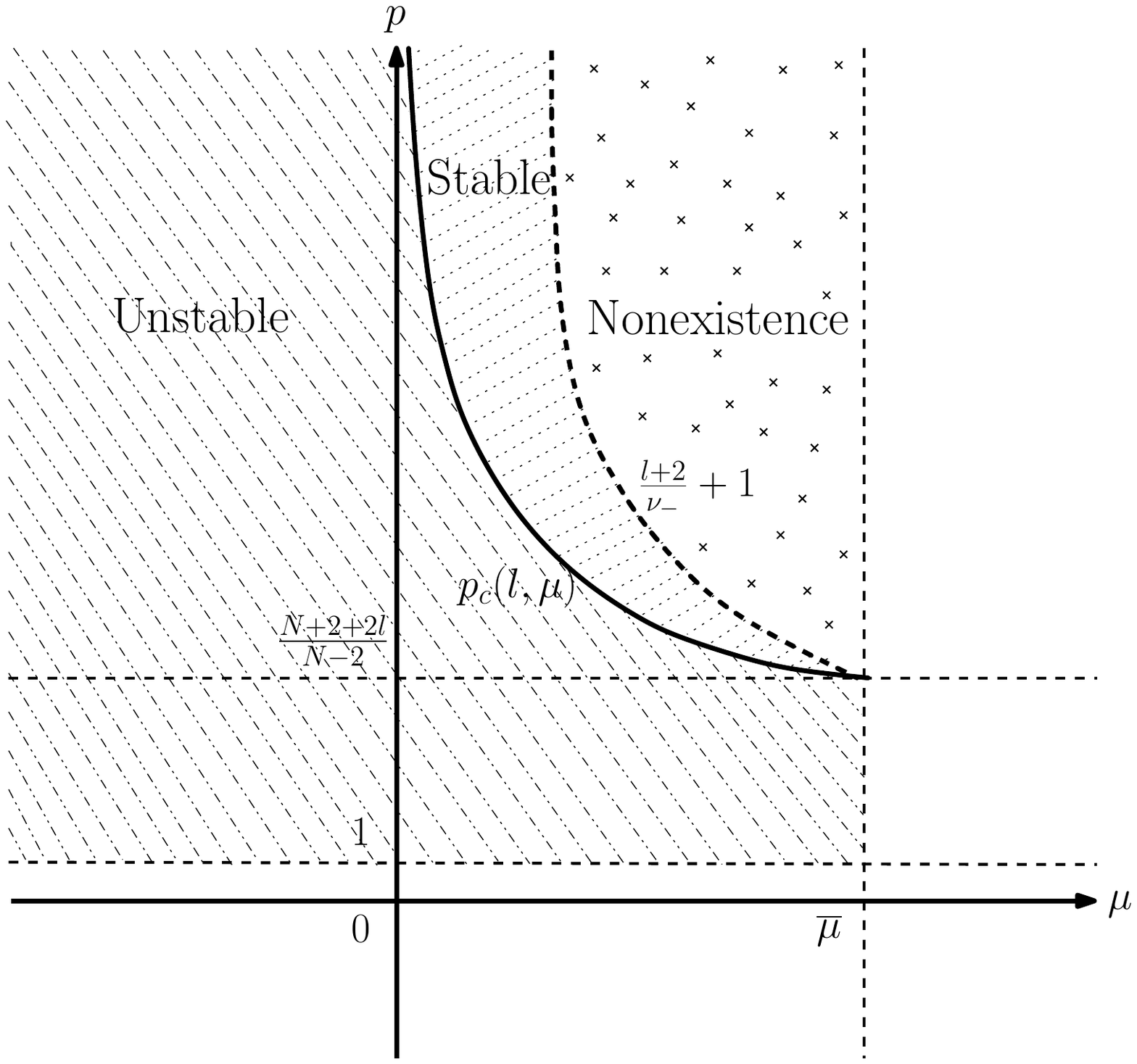}}
\caption{The case $N\leq 10+4l$}
\end{figure}
\begin{figure}[hbt]
\centering
\centerline{\includegraphics[scale = 0.6]{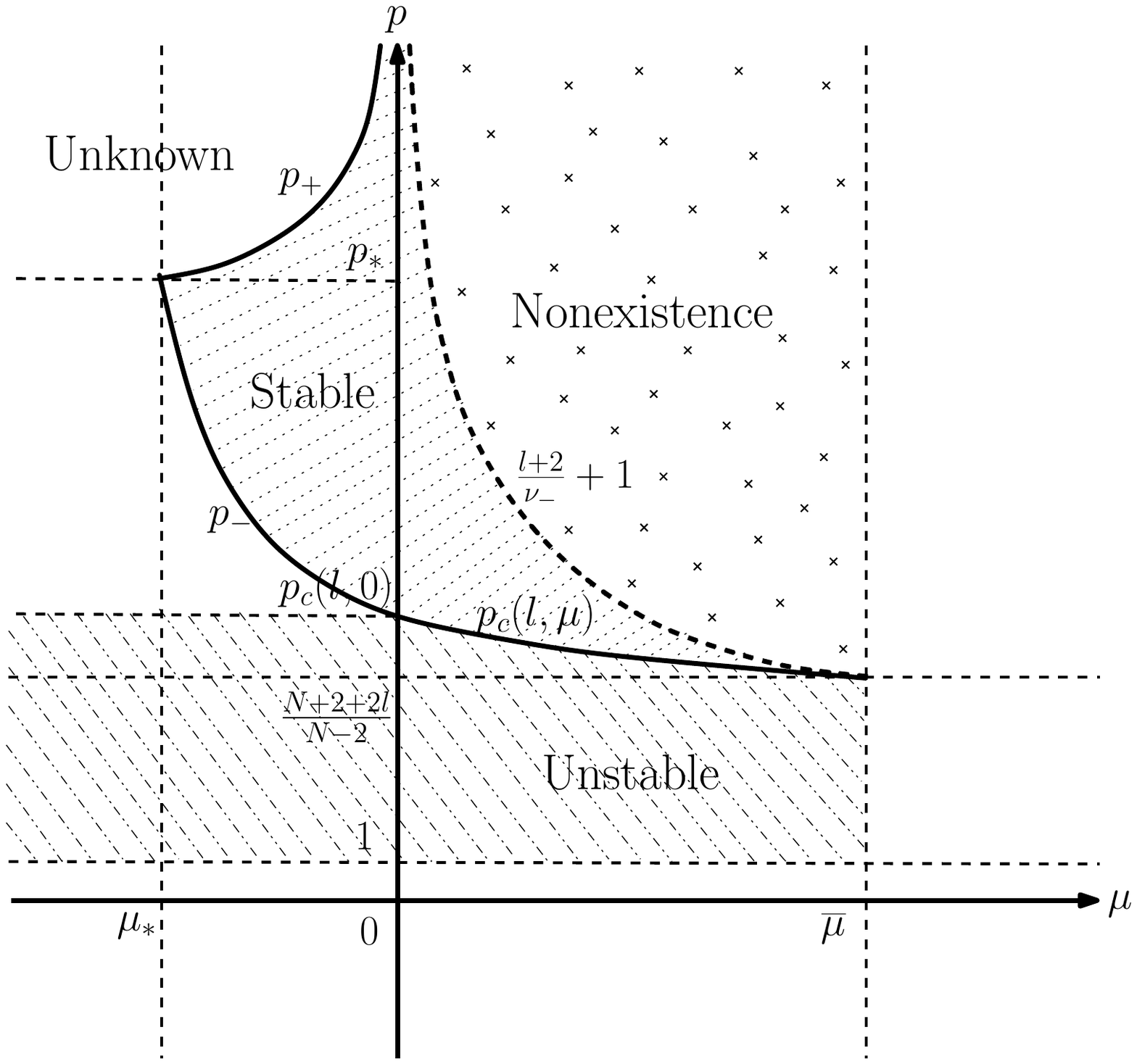}}
\caption{The case $N> 10+4l$}
\end{figure}

The next theorems say that there is no nontrivial finite Morse index
weak solution of (\ref{H}) in $\RN\setminus\{0\}$
if $1<p<p_c(l_-,\mu)$ with $l_-=\min\{l,0\}$ and $p\neq(N+2+2l)/(N-2)$.
The same result still holds in $\RN$.
We denote by $B_R(0)$ the open ball centered at zero with radius $R$ in $\RN$.

\begin{thm}\label{thm-morse1}
Suppose that $u$ is a finite Morse index weak solution of {\em \eqref{H}} in $\Omega=\RN\setminus\{0\}$.
Assume that
\begin{equation}\label{thm2.2-condition}
\int_{B_R(0)}|\nabla u|^2+|x|^l|u|^{p+1}\,dx<+\infty \qquad \textrm{for some}\ \ R>0.
\end{equation}
If $1<p<(N+2+2l)/(N-2)$, then $u\equiv 0$.
\end{thm}

In the following case, the condition \eqref{thm2.2-condition} is not needed.

\begin{cor}\label{cor-morse1}
If $u$ is a finite Morse index weak solution of \eqref{H} in $\Omega=\RN$
with $1<p<(N+2+2l)/(N-2)$, then $u\equiv 0$.
\end{cor}

\begin{thm}\label{thm-morse2}
Suppose that $u$ is a finite Morse index weak solution of \eqref{H} in $\Omega=\RN\setminus\{0\}$, $\RN$.
Assume that $(N+2+2l)/(N-2)< p < p_c(l_-,\mu)$
with $l_-=\min\{l,0\}$.
Then $u \equiv 0$.

Moreover, when $p=(N+2+2l)/(N-2)$,
every finite Morse index weak solution $u$ of \eqref{H} in $\RN\setminus\{0\}$ satisfies
\begin{eqnarray}
\label{thm2.3-limit1} \lim_{|x|\to 0}|x|^{\frac{l+2}{p-1}}|u(x)|=
\lim_{|x|\to 0}|x|^{1+\frac{l+2}{p-1}}|\nabla u(x)|=0, \\
\label{thm2.3-limit2} \lim_{|x|\to +\infty}|x|^{\frac{l+2}{p-1}}|u(x)|=
\lim_{|x|\to +\infty}|x|^{1+\frac{l+2}{p-1}}|\nabla u(x)|=0.
\end{eqnarray}
\end{thm}

On the other hand, we investigate the existence of a stable weak solution of (\ref{H}) in $\RN$.
For this purpose, we consider the following set:
\begin{equation*}
\begin{aligned}
S:=\Big{\{}&(\mu,p)\in(-\infty,\hmu)\times(1,+\infty)\ : \\
&\ p_c(l,\mu)\leq p<\dfrac{l+2}{\nu_-}+1 \qquad \textrm{if}\ \
0<\mu<\hmu\ \ \textrm{and}\ \ N\geq3, \\
&\ p_{-}\leq p \leq p_{+} \qquad \textrm{if}\ \ \mu_{\ast}\leq\mu\leq0\ \ \textrm{and}\ \ N>10+4l \Big{\}},
\end{aligned}
\end{equation*}
where
\begin{enumerate}[(i)]
\item $\nu_-:=\left[N-2-\sqrt{(N-2)^2-4\mu}\right]/2=\sqrt{\hmu}-\sqrt{\hmu-\mu}$ is one root of the equation
$\mu=\nu(N-2-\nu)$ for $\nu$;
\item $\mu_{\ast}:=-(2N+l-2)(N-10-4l)^2/108(l+2)$;
\item $p_{\pm}$ are solutions of the following equation for $p>1$:
$$\mu=\dfrac{l+2}{p-1}\left(N-2-\dfrac{l+2}{p-1}\right)
-\dfrac{1}{4(p-1)}\left(N-2-2\,\dfrac{l+2}{p-1}\right)^2$$
when $\mu_{\ast}\leq\mu\leq0$,
and two curves $p=p_-(\mu)$, $p_+(\mu)$ are decreasing, increasing in $\mu\in[\mu_{\ast},0]$ respectively;
\item the curves $p=p_{\pm}$ satisfy
$$\lim_{\mu\to 0^-}p_{-}=p_c(l,0), \qquad \lim_{\mu\to 0^-}p_{+}=+\infty,
\qquad \lim_{\mu\to{\mu_{\ast}}^+}p_{\pm}=p_{\ast}:=\dfrac{N+2+2l}{N-10-4l}.$$
\end{enumerate}
The dashed regions (``Stable" regions) in Figure 1 and Figure 2 describe the set $S$ in the $(\mu,p)$-plane.
Then we observe the following existence result.

\begin{thm}\label{thm-exist}
The equation \eqref{H} admits a family of stable positive radial solutions in $W^{1,2}_{loc}(\RN)$
if $(\mu,p)\in S$.
\end{thm}

\begin{rem}\em
(a) The above theorem shows that the exponent $p=p_c(l,\mu)$ is very sharp when $\mu\geq0$.\\
(b) The upper bound $\frac{l+2}{\nu_-}+1$ of $p$ in case $\mu>0$ is reasonable.
Indeed, there is no non-negative weak solution of \eqref{H} in $B_R(0)\setminus\{0\}$
if $0<\mu\leq\hmu$ and $p\geq\frac{l+2}{\nu_-}+1$ (see Theorem 2 and Remark 1 in \cite{BDT}).\\
(c)  {
In fact, Theorem \ref{thm-exist} is given briefly in \cite{Bae}. However, we will provide a direct and rigorous proof
in Section 5.}
\end{rem}

Throughout this paper, we complete the division of $(\mu,p)$-plane
according to the non-existence and existence of stable solutions in $\RN$
in both of cases (i) $N\leq10+4l$ and (ii) $N>10+4l$ and $\mu\geq0$.
However for the remaining case $N>10+4l$ and $\mu<0$, especially ``Unknown" region in Figure 2,
we could not finish the work.
This is an open problem to be considered.

\begin{rem}\em
As we said in Section 1,
the equivalent problem \eqref{CKN} to \eqref{H} was recently investigated in \cite{DG1}.
For comparison, we here consider one of their results (see \cite[Theorem 1.9]{DG1}).
Set $N'=N+\theta$ and $\tau=\alpha-\theta$ and assume that $N'>2$ and $\tau>-2$.
Then it was proved in \cite{DG1} that
if $u$ is a stable solution of \eqref{CKN} with
\begin{align}\label{DG-range}
\left\{ \begin{array}{ll}
1<p<+\infty \ \ &\textrm{and}\ \ \ 2< N'\leq 10+4\tau,\\
1<p<p_c(N',\tau) \ \ &\textrm{and}\ \ \ N'>10+4\tau,
\end{array}\right.
\end{align}
$u$ is identically zero, and that if $p\geq p_c(N',\tau)$, \eqref{CKN} admits
a family of stable positive radial solutions in $\RN$.
This is a result corresponding to our results, Theorems \ref{thm-stable} and \ref{thm-exist}.
In our case, we focus on representing explicit ranges of $N$ and $p$ determining
the existence and non-existence of finite Morse solutions of \eqref{H}.
However we encountered some obstacles when we used the techniques in \cite{DDG} and \cite{F1},
and thus our work left the unsolved region in case $\mu<0$.
In contrast with our work, replacing $(N,\mu,l)$ in our calculations by $(N',\tau)$,
the authors in \cite{DG1} was able to apply the techniques of \cite{DDG} and \cite{F1} into \eqref{CKN} without any problems
and particularly resolved the case $\mu_{\ast}\leq\mu<0$ and $p_c(l,0)\leq p<p_-$ which we could not solve.
However, since their critical exponents and cutting dimension expressed by $N'$ and $\tau$ depend on $p$,
\eqref{DG-range} in the above result does not immediately give explicit ranges of $N$ and $p$
determining the behavior of finite Morse index solutions to \eqref{CKN}
unlike our results.
\end{rem}

\begin{rem}\em
We have considered the condition $l>-2$ in \eqref{H}.
The restriction on $l$ is quite natural in some sense.
Indeed, if $l\leq-2$ and $\mu\geq0$, then
\eqref{H} does not admit a positive solution in $\RN\setminus\{0\}$.
We obtain this result by applying the argument used in \cite{DDG}.
\end{rem}

\begin{rem}\em
 {
We should inform of the originality of Figures 1 and 2 in this paper.
In \cite{Bae0, Bae}, Bae found the regions in which
the separation of positive radial solutions of \eqref{H} with $N\geq1$ and $l\in\mathbb{R}$ occurs,
and he also described the regions with the graphs in $(\mu,p)$-plane (see Figures 1-3 in \cite{Bae}).
In order to represent the regions indicating non-existence of stable solutions for \eqref{H},
we borrow two of them corresponding to the case $N\geq3$ and $l>-2$.
Moreover, we mostly follow Bae's notations in \cite{Bae} for connection with his results.}
\end{rem}


\section{Non-existence of stable solutions}
\setcounter{equation}{0}

In this section, we present a non-existence result for stable weak solutions of (\ref{H}) in $\RN$.
To this end, we first show the following proposition
which is crucially required in proving some of our results,
especially Theorems \ref{thm-stable}, \ref{thm-morse1} and \ref{thm-morse2}.
The integral estimate given in Proposition \ref{prop1} is a corresponding result of \cite[Proposition 4]{F1},
but the Hardy potential bothers us when we try to follow the proof in \cite{F1}.
Here the major technique enabling us to overcome the difficulty is Hardy's inequality.
Thus we recall Hardy's inequality as follows:
\begin{quote}
\emph{For every $u\in W^{1,2}(\RN)$, $N\geq3$, we have $u/|x|\in L^2(\RN)$ and
$$\hmu\intr |x|^{-2}|u|^2\,dx \leq \intr |\nabla u|^2\,dx,$$
where the constant $\hmu=(N-2)^2/4$ is optimal and not attained in $W^{1,2}(\RN)$.}
\end{quote}
In addition, we set $\gam_M(p,\mu)$ depending on $p$, $\mu$ and $N$ given by
\begin{equation}\label{gam max}
\gam_M(p,\mu):=\dfrac{(2\hmu-\mu_+)p+\mu_+ -\hmu+2\sqrt{\hmu(\hmu-\mu_+)p(p-1)}}{(p-1)\mu_+ +\hmu}
\end{equation}
with $\mu_+=\max\{\mu,0\}$.
Note that it is well-defined (because $\mu<\hmu$)
and $\gam_M(p,\mu)=\gam_M(p,0)=2p+2\sqrt{p(p-1)}-1$ when $\mu\leq 0$.
We point out that $\gam_M(p,\mu)$ will come out from the proof of Proposition \ref{prop1}
and it plays a crucial role in our proofs.

A full proof of Proposition \ref{prop1} will be given here.
We remark that in Proposition \ref{prop1}, we do not need any restriction on $l$.

\begin{prop}\label{prop1}
Let $\Omega$ be a domain (bounded or not) of $\RN$.
Suppose that $u$ is a stable weak solution of \eqref{H} in $\Omega$.
Then, for any integer $m \geq {\rm max}\{\frac{p+\gam}{p-1}, 2\}$
and any $\gam\in[1,\gam_M(p,\mu))$,
there exists a constant $C>0$  such that
$$\into \left(\left|\nabla \left(|u|^{\frac{\gamma-1}{2}}u\right)\right|^2
+|x|^l|u|^{p+\gamma}\right)|\psi|^{2m}\,dx
\leq C \into |x|^{\frac{(\gamma+1)l}{1-p}}
\left(|\nabla\psi|^2+|\psi\Delta\psi|\right)^{\frac{p+\gamma}{p-1}}\,dx$$
for all test functions $\psi \in C^2_c(\Omega)$ satisfying $|\psi|\leq1$ in $\Omega$.
\end{prop}

\begin{proof}
We divide the proof into four steps.

\noindent\emph{Step 1. For any $\vphi \in C^2_c(\Omega)$,
\begin{equation}
\begin{aligned}\label{prop2.4-step1}
\into\left|\nabla \left(|u|^{\frac{\gamma-1}{2}}u\right)\right|^2 \vphi^2 \,dx
\ =\ & \dfrac{(\gamma+1)^2}{4\gamma}\int_{\Omega}\left(|x|^l |u|^{p+\gamma}\vphi^2
+\mu|x|^{-2}|u|^{\gamma+1}\vphi^2\right)\,dx \\
&+\ \dfrac{\gamma+1}{4\gamma}\int_{\Omega}|u|^{\gamma+1}\Delta\vphi^2\,dx.
\end{aligned}
\end{equation}}

Taking $\phi=|u|^{\gam-1}u\vphi^2\in W^{1,2}_{loc}(\Omega)\cap L^{\infty}(\Omega)$
in (\ref{def-wk sol}), we obtain the identity (\ref{prop2.4-step1}).
\\
\emph{Step 2. For any $\vphi \in C^2_c(\Omega)$,
\begin{equation}
\begin{aligned}\label{prop2.4-step2}
\left(p -\dfrac{(\gam+1)^2}{4\gam}\right)\into |x|^l |u|^{p+\gam} \vphi^2\,dx
\ \leq\ & \into |u|^{\gam+1}|\nabla \vphi|^2\,dx
+\dfrac{1-\gam}{4\gam} \into |u|^{\gam+1} \Delta \vphi^2\,dx \\
&+\ \mu\,\dfrac{(\gam-1)^2}{4\gam}\into|x|^{-2}|u|^{\gam+1}\vphi^2\,dx.
\end{aligned}
\end{equation}}

We choose a function $\phi=|u|^{\frac{\gam-1}{2}}u\vphi$
with $\vphi\in C^2_c(\Omega)$, then $\phi\in W^{1,2}_{loc}(\Omega)\cap L^{\infty}(\Omega)$
and it can be used as a test function in (\ref{def-stable}).
Using the stability of $u$, we have
\begin{eqnarray*}
p \into |x|^l |u|^{p+\gam} \vphi^2\,dx
&\leq& \into \left|\nabla \left(|u|^{\frac{\gamma-1}{2}}u\right)\right|^2 \vphi^2\,dx
+ \into |u|^{\gam+1}|\nabla \vphi|^2\,dx \\
& &-\dfrac{1}{2} \into |u|^{\gam+1} \Delta \vphi^2\,dx
- \mu \into |x|^{-2} |u|^{\gam+1} \vphi^2\,dx.
\end{eqnarray*}
It then follows from (\ref{prop2.4-step1}) that
\begin{eqnarray*}
p \into |x|^l |u|^{p+\gam} \vphi^2\,dx
&\leq& \dfrac{(\gam+1)^2}{4\gam}\into|x|^l |u|^{p+\gam}\vphi^2\,dx
+ \into |u|^{\gam+1}|\nabla \vphi|^2\,dx \\
& &+\dfrac{1-\gam}{4\gam} \into |u|^{\gam+1} \Delta \vphi^2\,dx
+ \mu\dfrac{(\gam-1)^2}{4\gam} \into |x|^{-2}|u|^{\gam+1} \vphi^2\,dx,
\end{eqnarray*}
which implies the inequality (\ref{prop2.4-step2}).
\\
\emph{Step 3. We consider two cases $\mu\leq0$ and $\mu>0$, separately.
\begin{enumerate}[{\em (i)}]
\item If $\mu\leq0$, then for any $\gam \in [1, 2p+2\sqrt{p(p-1)}-1)$,
\begin{equation}\label{prop2.4-step3-i}
\left(p-\dfrac{(\gam+1)^2}{4\gam}\right) \into |x|^l |u|^{p+\gamma}\vphi^2\,dx
\leq \into |u|^{\gam+1}|\nabla \vphi|^2\,dx
+ \dfrac{1-\gam}{4\gam} \into |u|^{\gam+1}|\vphi\Delta\vphi|\,dx,
\end{equation}
where $$p-\dfrac{(\gam+1)^2}{4\gam}>0.$$
\item If $\mu>0$, then for any $\gam \in [1,
\frac{(2\hmu-\mu)p+\mu-\hmu+2\sqrt{\hmu(\hmu-\mu)p(p-1)}}{(p-1)\mu+\hmu})$,
\begin{equation}
\begin{aligned}\label{prop2.4-step3-ii}
&\left(p-\dfrac{(\gam+1)^2}{4\gam}-\dfrac{(\gam-1)^2(\gam+1)^2}{16\gam^2\alpha}\right)
\into |x|^l |u|^{p+\gam} \vphi^2 \,dx \\
&\leq\ \left(1+\dfrac{(\gam-1)^2}{4\gam\alpha}\right)
\left\{\into |u|^{\gam+1}|\nabla\vphi|^2\,dx
+ \dfrac{1-\gam}{4\gam}\into |u|^{\gam+1}|\vphi\Delta\vphi|\,dx \right\},
\end{aligned}
\end{equation}
where $$\alpha=\frac{\hmu}{\mu}-\frac{(\gam+1)^2}{4\gam}>0 \qquad \textrm{and} \qquad
p-\dfrac{(\gam+1)^2}{4\gam}-\dfrac{(\gam-1)^2(\gam+1)^2}{16\gam^2\alpha}>0.$$
\end{enumerate}}

In case $\mu\leq0$, (\ref{prop2.4-step2}) immediately yields
the desired inequality (\ref{prop2.4-step3-i}) with
$$p-\dfrac{(\gam+1)^2}{4\gam}>0 \qquad \textrm{for all}\ \ \gam\in[1,2p+2\sqrt{p(p-1)}-1).$$

On the other hand, we get from Hardy's inequality and (\ref{prop2.4-step1}) that
\begin{eqnarray*}
\hmu \into |x|^{-2}|u|^{\gam+1}\vphi^2 \,dx
&\leq& \into \left|\nabla\left(|u|^{\frac{\gam-1}{2}}u\vphi\right)\right|^2\,dx \\
&=& \into\left|\nabla\left(|u|^{\frac{\gam-1}{2}}u\right)\right|^2\vphi^2\,dx +\into |u|^{\gam+1}|\nabla\vphi|^2\,dx
- \dfrac{1}{2}\into|u|^{\gam+1}\Delta\vphi^2\,dx \\
&=&\dfrac{(\gam+1)^2}{4\gam} \into\left(|x|^l |u|^{p+\gam}\vphi^2
+ \mu|x|^{-2}|u|^{\gam+1}\vphi^2\right)\,dx \\
& &+\into |u|^{\gam+1}|\nabla \vphi|^2\,dx
+\dfrac{1-\gam}{4\gam}\into |u|^{\gam+1}\Delta\vphi^2\,dx.
\end{eqnarray*}
This implies that
\begin{equation*}
\begin{aligned}
&\left(\dfrac{\hmu}{\mu}-\dfrac{(\gam+1)^2}{4\gam}\right)
\mu\into|x|^{-2}|u|^{\gam+1}\vphi^2\,dx \\
&\ \ \leq \dfrac{(\gam+1)^2}{4\gam} \into |x|^l |u|^{p+\gam} \vphi^2\,dx
+\into |u|^{\gam+1}|\nabla \vphi|^2\,dx
+\dfrac{1-\gam}{4\gam} \into |u|^{\gam+1} \Delta \vphi^2\,dx.
\end{aligned}
\end{equation*}
Then for any $\mu\in(0,\hmu)$, a direct computation shows that
$$\alpha=\dfrac{\hmu}{\mu}-\dfrac{(\gam+1)^2}{4\gam}>0 \qquad
\textrm{if}\ \ 1\leq \gam < \dfrac{2\hmu-\mu+2\sqrt{\hmu(\hmu-\mu)}}{\mu}$$
and
$$\gam_M(p,\mu)=\dfrac{(2\hmu-\mu)p+\mu-\hmu+2\sqrt{\hmu(\hmu-\mu)p(p-1)}}{(p-1)\mu+\hmu}
<\dfrac{2\hmu-\mu+2\sqrt{\hmu(\hmu-\mu)}}{\mu},$$
and thus $\alpha>0$ for all $\gam\in[1,\gam_M(p,\mu))$.
Therefore we have that for any $\mu\in(0,\hmu)$ and $\gam\in[1,\gam_M(p,\mu))$,
\begin{equation}
\begin{aligned}\label{prop2.4-H term}
&\mu\into|x|^{-2}|u|^{\gam+1}\vphi^2\,dx \\
&\ \ \leq\dfrac{1}{\alpha}\left\{\dfrac{(\gam+1)^2}{4\gam}\into |x|^l |u|^{p+\gam} \vphi^2\,dx
+\into |u|^{\gam+1}|\nabla \vphi|^2\,dx
+\dfrac{1-\gam}{4\gam} \into |u|^{\gam+1} \Delta \vphi^2\,dx\right\}.
\end{aligned}
\end{equation}
Applying (\ref{prop2.4-H term}) into (\ref{prop2.4-step2}), we get
\begin{equation*}
\begin{aligned}
&\left(p-\dfrac{(\gam+1)^2}{4\gam}\right) \into |x|^l |u|^{p+\gam}\vphi^2\,dx \\
&\ \ \leq \dfrac{(\gam-1)^2}{4\gam\alpha} \left\{
\dfrac{(\gam+1)^2}{4\gam} \into |x|^l |u|^{p+\gam}\vphi^2\,dx
+\into |u|^{\gam+1} |\nabla\vphi|^2\,dx
+\dfrac{1-\gam}{4\gam} \into |u|^{\gam+1}\Delta\vphi^2\,dx \right\} \\
&\ \ \ \ \ +\into |u|^{\gam+1} |\nabla\vphi|^2\,dx
+\dfrac{1-\gam}{4\gam} \into |u|^{\gam+1} \Delta\vphi^2\,dx,
\end{aligned}
\end{equation*}
where
\begin{equation*}
\beta:=p-\dfrac{(\gam+1)^2}{4\gam}-\dfrac{(\gam-1)^2(\gam+1)^2}{16\gam^2\alpha} > 0
\qquad \textrm{for all}\ \ \gam\in[1,\gam_M(p,\mu)).
\end{equation*}
We complete the proof of (\ref{prop2.4-step3-ii}).
\\
\emph{Step 4. End of proof.}

For any $\psi\in C^2_c(\Omega)$ with $|\psi|\leq 1$,
we insert the function $\vphi=\psi^m\in C^2_c(\Omega)$
into (\ref{prop2.4-step3-i}) and (\ref{prop2.4-step3-ii}) respectively.
Then it follows that
\begin{equation}\label{prop2.4-step4-1}
\into |x|^l |u|^{p+\gam} |\psi|^{2m}\,dx
\leq C_1 \into |u|^{\gam+1}|\psi|^{2m-2}\left(|\nabla\psi|^2
+ |\psi\Delta\psi|\right)\,dx
\end{equation}
for some positive constant $C_1$.
Using H\"{o}lder's inequality we obtain
\begin{equation*}
\begin{aligned}
\into |x|^l |u|^{p+\gam} |\psi|^{2m}\,dx
\leq\ &C_1 \left\{\into \left(|x|^{\frac{(\gam+1)l}{p+\gam}} |u|^{\gam+1} |\psi|^{2m-2}\right)
^{\frac{p+\gam}{\gam+1}}\,dx\right\}^{\frac{\gam+1}{p+\gam}} \\
&\times\left\{\into \left(|x|^{-\frac{(\gam+1)l}{p+\gam}}
\left(|\nabla\psi|^2 + |\psi\Delta\psi|\right)\right)
^{\frac{p+\gam}{p-1}}\,dx\right\}^{\frac{p-1}{p+\gam}}.
\end{aligned}
\end{equation*}
Here we easily see that
$m \geq {\rm max}\left\{\frac{p+\gamma}{p-1},\ 2\right\}$ implies
$2(m-1)\,\frac{p+\gam}{\gam+1} \geq 2m$,
so that $|\psi|^{2(m-1)\frac{p+\gam}{\gam+1}} \leq |\psi|^{2m}$
(because $|\psi| \leq 1$).
Hence there exists a constant $C>0$ such that
\begin{eqnarray}\label{prop2.4-step4-2}
\into |x|^l |u|^{p+\gam} |\psi|^{2m}\,dx
\leq C\into |x|^{\frac{(\gam+1)l}{1-p}}
\left(|\nabla\psi|^2 + |\psi\Delta\psi|\right)^{\frac{p+\gam}{p-1}}\,dx.
\end{eqnarray}

On the other hand, when $\mu\leq0$, we combine (\ref{prop2.4-step1}) and (\ref{prop2.4-step3-i}) to get
\begin{equation}\label{prop2.4-step4-3}
\into \left|\nabla \left(|u|^{\frac{\gamma-1}{2}}u\right)\right|^2\vphi^2\,dx
\leq A_1\into|u|^{\gam+1}|\nabla\vphi|^2\,dx+B_1\into|u|^{\gam+1}|\vphi\Delta\vphi|\,dx
\end{equation}
with positive constants $A_1$ and $B_1$,
where we just use the fact that $\mu\leq0$.
When $\mu>0$, applying (\ref{prop2.4-H term}) into (\ref{prop2.4-step1}) we obtain
\begin{equation*}
\begin{aligned}
\into \left|\nabla \left(|u|^{\frac{\gamma-1}{2}}u\right)\right|^2\vphi^2\,dx
\ \leq\ & \dfrac{(\gam+1)^2}{4\gam}\left(1+\dfrac{(\gam+1)^2}{4\gam\alpha}\right)
\into |x|^l|u|^{p+\gam}\vphi^2\,dx \\
& +\dfrac{\gam+1}{2\gam}\left(\dfrac{(\gam+1)^2}{4\gam\alpha}+1\right)\into |u|^{\gam+1}|\nabla\vphi|^2\,dx \\
& +\left(\dfrac{(\gam+1)^2}{4\gam\alpha}\dfrac{1-\gam}{2\gam}+\frac{\gam+1}{2\gam}\right)
\into|u|^{\gam+1}|\vphi\Delta\vphi|\,dx,
\end{aligned}
\end{equation*}
and then we deduce from (\ref{prop2.4-step3-ii}) that
\begin{equation}\label{prop2.4-step4-4}
\into \left|\nabla \left(|u|^{\frac{\gamma-1}{2}}u\right)\right|^2\vphi^2\,dx
\leq A_2\into|u|^{\gam+1}|\nabla\vphi|^2\,dx+B_2\into|u|^{\gam+1}|\vphi\Delta\vphi|\,dx,
\end{equation}
where $A_2$ and $B_2$ are positive constants.

Now we insert again the test function $\vphi=\psi^m\in C^2_c(\Omega)$
into (\ref{prop2.4-step4-3}) and (\ref{prop2.4-step4-4}) and we find
\begin{equation*}
\into \left|\nabla \left(|u|^{\frac{\gamma-1}{2}}u\right)\right|^2|\psi|^{2m}\,dx
\leq  C_2 \into |u|^{\gam+1}|\psi|^{2m-2}\left(|\nabla\psi|^2
+ |\psi\Delta\psi|\right)\,dx
\end{equation*}
for some positive constant $C_2$.
Using H\"{o}lder's inequality and (\ref{prop2.4-step4-2}) leads to
\begin{eqnarray}\label{prop2.4-step4-5}
\into \left|\nabla \left(|u|^{\frac{\gamma-1}{2}}u\right)\right|^2|\psi|^{2m}\,dx
&\leq& C_2 \left\{\into |x|^l |u|^{p+\gam}|\psi|^{2m}\,dx\right\}^{\frac{\gam+1}{p+\gam}} \nonumber\\
& &\ \times\left\{\into |x|^{\frac{(\gam+1)l}{1-p}} \left(|\nabla\psi|^2 + |\psi\Delta\psi|\right)
^{\frac{p+\gam}{p-1}}\,dx\right\}^{\frac{p-1}{p+\gam}} \nonumber\\
&\leq& C \into |x|^{\frac{(\gam+1)l}{1-p}}
\left(|\nabla\psi|^2 + |\psi\Delta\psi|\right)^{\frac{p+\gam}{p-1}}\,dx.
\end{eqnarray}
Adding (\ref{prop2.4-step4-2}) and (\ref{prop2.4-step4-5}) immediately yields the desired inequality.
\end{proof}

\begin{rem}\label{rem-sol}\em
We can show that Proposition \ref{prop1} still holds for a weak solution $u$ without local boundedness,
which is introduced in Remark \ref{wk sol},
by using truncations of $u$ as in \cite[Proposition 3.1]{WY}.
Indeed, let $\zeta_k(t)=\max(-k, \min(t,k))$ with $k\in\mathbb{N}$ and
use the test function $|\zeta_k(u)|^{\frac{\gam-1}{2}}u\varphi$ in \eqref{def-wk sol}
with $\varphi\in C^2_c(\Omega)$.
Then the rest of proof can be proceeded as the above proof
and finally we take $k$ tending to $\infty$ (see \cite{WY} for details).
\end{rem}

\begin{proof}[\textbf{Proof of Theorem \ref{thm-stable}}]
Suppose that $u$ is a stable weak solution of (\ref{H}) in $\RN$.
For any $R>0$, we set $\psi_R(x)=\psi(|x|/R)$
with $\psi\in C^2_c (\mathbb{R})$ satisfying $0\le\psi\le1$ in $\mathbb{R}$ and
\begin{align*}
&\psi(t)=\left \{ \begin{array}{ll}
1 \qquad \textrm{if}\ \ |t|\le1, \\
0 \qquad \textrm{if}\ \ |t|\ge2,
\end{array}\right.
\end{align*}
and we use the function $\psi_R$ as a test function in Proposition \ref{prop1}.
Then for any integer $m \geq \max\{\frac{p+\gam}{p-1}, 2\}$
and any $\gam\in[1,\gam_M(p,\mu))$,
\begin{eqnarray}\label{stable-int}
\int_{B_R(0)} |x|^l |u|^{p+\gamma}\,dx
&\le& C_1 \int_{B_{2R}(0)\setminus B_R(0)} |x|^{\frac{(\gamma+1)l}{1-p}}
\left(|\nabla\psi_R|^2+|\psi_R \Delta\psi_R|\right)^{\frac{p+\gamma}{p-1}}\,dx \nonumber\\
&\le& CR^{N-\frac{(\gamma+1)l+2(p+\gam)}{p-1}}
\qquad \textrm{for all}\ \ R>0,
\end{eqnarray}
where $C$, $C_1$ are positive constants independent of $R$.

Fix $N\geq3$ and $l>-2$.
Then we claim that, under the assumptions on $p$ and $\mu$ in Theorem \ref{thm-stable},
we can always choose $\gam\in[1,\gam_M(p,\mu))$ such that
\begin{equation}\label{show}
N-\frac{(\gamma+1)l+2(p+\gamma)}{p-1}<0.
\end{equation}
We consider the real-valued function
$f(p,\mu)$ on $(1,\infty)\times(-\infty,\hmu)$ defined by
\begin{equation}
\begin{aligned}\label{fn f}
f(p,\mu)\ &:=\ \dfrac{2p+l+(l+2)\gam_M(p,\mu)}{p-1} \\
&=\ 2+\dfrac{2(l+2)}{(p-1)\mu_+ +\hmu}\left\{\hmu+\dfrac{\hmu}{p-1}
+\sqrt{\hmu(\hmu-\mu_+)\left(1+\dfrac{1}{p-1}\right)}\right\}
\end{aligned}
\end{equation}
with $\mu_+=\max\{\mu,0\}$.
Then, by continuity of the function $t\mapsto N-\frac{2p+l+(l+2)t}{p-1}$,
we find the fact that $N-f(p,\mu)<0$ implies (\ref{show}).
Moreover since $f(p,\mu)$ is a strictly decreasing function with respect to $p$,
it holds that $N-f(p,\mu)<0$ for $1<p<p_c(l,\mu)$,
where $p_c(l,\mu)$ is the unique exponent satisfying
\begin{equation}\label{show1}
N-f(p_c(l,\mu),\mu)=0 \qquad \textrm{and} \qquad p_c(l,\mu)>1.
\end{equation}
Therefore, it suffices to find the exponent $p_c(l,\mu)$ satisfying (\ref{show1}).
The cases $\mu\leq0$ and $\mu>0$ will be treated separately.

\noindent \textbf{Case 1) $\mu\leq0$:}
The proof of this case is exactly the same as in \cite{DDG},
so we omit the details.

\noindent \textbf{Case 2) $\mu>0$:}
We observe that $f(p,\mu)$ is a strictly decreasing function in $p$ satisfying
$$\lim_{p\rightarrow 1^+}f(p,\mu)=+\infty
\qquad \textrm{and} \qquad \lim_{p\rightarrow+\infty}f(p,\mu)=2.$$
Then for each $\mu>0$,
there is a unique solution $p_c(l,\mu)>1$ of (\ref{show1}).

Now, if $1<p<p_c(l,\mu)$, by letting $R\to+\infty$ in (\ref{stable-int})
with $\gam$ satisfying (\ref{show}), we get
$$\intr |x|^l |u|^{p+\gamma}\,dx=0,$$
and so $u\equiv0$.

We turn to the properties of $p_c(l,\mu)$ for $0<\mu<\hmu$.
We claim that $p_c(l,\mu)$ is a strictly decreasing function in $\mu$
satisfying (\ref{p_c limit}).
Indeed, since $\gam_M(p,\mu)$ is a strictly decreasing function in $\mu$,
we have
\begin{equation}\label{der of f}
\dfrac{\partial f}{\partial\mu}(p,\mu)
=\dfrac{l+2}{p-1}\,\dfrac{\partial\gam_M}{\partial\mu}(p,\mu)<0
\qquad \textrm{for all}\ \ p>1.
\end{equation}
Moreover, differentiating both side of $f(p_c(l,\mu),\mu)=N$
with respect to $\mu$ we see that
$$\dfrac{\partial f}{\partial\mu}(p,\mu)
+\dfrac{\partial f}{\partial p}(p,\mu)\,
\dfrac{\partial p_c(l,\mu)}{\partial\mu}\Bigg{|}_{p=p_c(l,\mu)}=0.$$
Because of \eqref{der of f} and $\partial f/\partial p<0$,
we obtain
$$\dfrac{\partial p_c(l,\mu)}{\partial\mu}<0$$
and thus $p_c(l,\mu)$ is a strictly decreasing function in $\mu$.
Furthermore, by the implicit function theorem,
we find that $p=p_c(l,\mu)$ is a $C^1$ curve with respect to $\mu$.
Therefore the limit behavior of $p_c(l,\mu)$ satisfies (\ref{p_c limit});
here we use the fact that $(N+2+2l)/(N-2)$ is
a solution of $N=f(p,\hmu)=2(p+l+1)/(p-1)$.

In addition, when $l\geq0$, we see that $p=p_c(l,\mu)$ is strictly
increasing in $l$ by the above argument.
Indeed, with the function $f=f(p,\mu,l)$ in (\ref{fn f}),
we differentiate both side of $f(p_c(l,\mu),\mu_+,l)=N$
with respect to $l$ and then obtain the desired result.

We complete the proof of Theorem \ref{thm-stable}.
\end{proof}


\section{Non-existence of finite Morse index solutions}
\setcounter{equation}{0}

In this section, we prove the non-existence of
nontrivial finite Morse index weak solutions of (\ref{H}) in $\Omega=\RN$, $\RN\setminus\{0\}$.
We shall apply Pohozaev's identity in the proof of Theorem \ref{thm-morse1}
and use the behavior of finite Morse index weak solutions
near the origin and infinity in the proof of Theorem \ref{thm-morse2}.

Let us start with some facts.
Suppose that a weak solution $u$ of (\ref{H}) has finite Morse index in $\RN\setminus\{0\}$.
Then $u$ is stable outside a compact subset of $\RN\setminus\{0\}$,
and thus there exists a $R_0>0$ such that
$u$ is stable in $\RN\setminus B_{R_0}(0)$.
By the similar way, $u$ is also stable in $B_{\ep_0}(0)\setminus\{0\}$ for some $\ep_0>0$.
Moreover, finite Morse index weak solutions of (\ref{H}) in $\RN$ obviously satisfy
the stability condition over a punctured ball and an exterior domain
(see also \cite[Proposition 2.1]{DDF}).
It leads to the non-existence result for nontrivial finite Morse index weak solutions
of (\ref{H}) in $\RN$ by the same argument as in Theorems \ref{thm-morse1} and \ref{thm-morse2}
(see Corollary \ref{cor-morse1} and the case $\Omega=\RN$ in Theorem \ref{thm-morse2}).

The next lemmas are crucial steps in the proofs of Theorems \ref{thm-morse1} and \ref{thm-morse2},
and here we use Proposition \ref{prop1}.

\begin{lem}\label{lem1}
Suppose that a weak solution $u$ of \eqref{H} in $\RN\setminus\{0\}$
is stable in $B_{\ep_0}(0)\setminus\{0\}$.
For any $\gam\in[1,\gam_M(p,\mu))$, we have that
\begin{enumerate}[\em (a)]
\item there exists a $\ep_{\ast}\in(0,\ep_0)$ such that
for any $\ep\in(0,\ep_{\ast}/2)$,
\begin{equation}\label{lem3.1-a}
\int_{\{\ep<|x|<\ep_{\ast}\}}
\left|\nabla \left(|u|^{\frac{\gamma-1}{2}}u\right)\right|^2
+ |x|^l |u|^{p+\gamma}\,dx
\le C_1+C_2\,\ep^{N-\frac{(\gamma+1)l+2(p+\gam)}{p-1}},
\end{equation}
where $C_1$ and $C_2$ are positive constants, independent of $\ep$;
\item for every open ball $B_{\rho}(y)$ with
$0<|y|<2\ep_0/3$ and $\rho=|y|/4$,
\begin{equation}\label{lem3.1-b}
\int_{B_{\rho}(y)} \left|\nabla \left(|u|^{\frac{\gamma-1}{2}}u\right)\right|^2
+ |x|^l |u|^{p+\gamma}\,dx
\le C\,\rho^{N-\frac{(\gamma+1)l+2(p+\gam)}{p-1}},
\end{equation}
where $C$ is a positive constant,
independent of $\rho$ and $y$.
\end{enumerate}
\end{lem}

\begin{proof}
The proof is based on the argument introduced in \cite{F1}
and it is similar to that in \cite{DDG}
(see Step 1 of proofs for Theorems 2.1 and 2.2 in \cite{DDG}).
Thus we skip it here and refer to \cite{DDG,F1} for further details.
\end{proof}

\begin{lem}\label{lem2}
Suppose that a weak solution $u$ of \eqref{H} in $\RN\setminus\{0\}$
is stable in $\RN\setminus B_{R_0}(0)$.
For any $\gam\in[1,\gam_M(p,\mu))$, we have that
\begin{enumerate}[\em (a)]
\item for every $r > R_0+3$,
\begin{equation}\label{lem3.2-a}
\int_{\{R_0+2<|x|<r\}}
\left|\nabla \left(|u|^{\frac{\gamma-1}{2}}u\right)\right|^2
+ |x|^l |u|^{p+\gamma}\,dx
\le C_1+C_2\,r^{N-\frac{(\gamma+1)l+2(p+\gam)}{p-1}},
\end{equation}
where $C_1$ and $C_2$ are positive constants,
independent of $r$;
\item for every open ball $B_R(y)$ with $|y|\geq 2R_0$
and $R=|y|/4$,
\begin{equation}\label{lem3.2-b}
\int_{B_R(y)} \left|\nabla \left(|u|^{\frac{\gamma-1}{2}}u\right)\right|^2
+ |x|^l |u|^{p+\gamma}\,dx
\le C\,R^{N-\frac{(\gamma+1)l+2(p+\gam)}{p-1}},
\end{equation}
where $C$ is a positive constant, independent of $R$ and $y$.
\end{enumerate}
\end{lem}

\begin{proof}
The proof is based on the argument introduced in \cite{F1}
and it is similar to that in \cite{DDG}
(see Step 1 of the proof for Theorem 3.3 in \cite{DDG}).
Thus we skip it here and refer to \cite{DDG,F1} for further details.
\end{proof}

\subsection{Proof of Theorem \ref{thm-morse1}}

Assume that $1<p<(N+2+2l)/(N-2)$.
Choosing $\gamma=1$ in Lemma \ref{lem2}
and using (\ref{thm2.2-condition}) we obtain
\begin{equation}\label{bdd}
\nabla u\in L^2(\RN) \qquad \textrm{and}\ \qquad \intr |x|^l|u|^{p+1}\,dx<\infty.
\end{equation}
Since $u\in C^2(\RN\setminus\{0\})$ by elliptic regularity theory,
we can apply Pohozaev identity to $u$ in $B_{\sigma,R}:=\{|x|\in\RN\,:\,\sigma<|x|<R\}$, and then we see that
\begin{equation}\label{pohozaev}
\begin{aligned}
&\frac{N+l}{p+1}\int_{B_{\sigma,R}}|x|^l|u|^{p+1}\,dx+
\frac{N-2}{2}\int_{B_{\sigma,R}}\mu|x|^{-2}u^2-|\nabla u|^2\,dx \\
&\ \ =\int_{\partial B_{\sigma,R}}\frac{(\nabla u\cdot x)^2}{|x|}-\frac{|\nabla u|^2|x|}{2}\,dS
+\frac{\mu}{2}\int_{\partial B_{\sigma,R}}|x|^{-1}u^2\,dS
+\frac{1}{p+1}\int_{\partial B_{\sigma,R}}|x|^{l+1}|u|^{p+1}\,dS.
\end{aligned}
\end{equation}
We will now show that the right hand side in \eqref{pohozaev} converges to $0$ for suitably chosen sequences
$\sigma_n\to0$ and $R_n\to+\infty$.
We first define
$$I(r):=\int_{\partial B_{r}}|\nabla u|^2+\mu|x|^{-2}u^2+
\frac{2}{p+1}|x|^l|u|^{p+1}\,dS.$$
Then by \eqref{bdd} and Hardy's inequality, we have that for any $c$ with $\sigma<c<R$,
$$\intr |\nabla u|^2+\mu|x|^{-2}u^2+\frac{2}{p+1}|x|^l|u|^{p+1}\,dx=
\int^{+\infty}_{c}I(R)\,dR +\int^{c}_{0}I(\sigma)\,d\sigma<+\infty.$$
Therefore there exist the sequences $\sigma_n\to0$ and $R_n\to+\infty$ such that
$\lim_{n\to\infty}\sigma_nI(\sigma_n)=0$ and $\lim_{n\to\infty}R_nI(R_n)=0$.
We derive from \eqref{pohozaev} (with the choice $\sigma=\sigma_n$, $R=R_n$ and $n\to+\infty$) that
\begin{equation}\label{pohozaev id}
\frac{N-2}{2}\intr|\nabla u|^2\,dx-\mu|x|^{-2}|u|^2\,dx
=\frac{l+N}{p+1}\intr|x|^l |u|^{p+1}\,dx.
\end{equation}

On the other hand, we consider the function $\vphi$ defined in the proof of Lemma \ref{lem1}
and set $\vphi_{R,0}(x):=\varphi(|x|/R)$.
Using $u\,\vphi_{R,0}$ as a test function in (\ref{def-wk sol})
and integrating by parts it follows that
\begin{equation}\label{thm2.2-1}
\intr\left(|\nabla u|^2-\mu|x|^{-2}|u|^2-|x|^l |u|^{p+1}\right)\vphi_{R,0}\,dx
=\frac{1}{2}\intr|u|^2\Delta\vphi_{R,0}\,dx.
\end{equation}
In particular, applying H\"{o}lder's inequality we find
\begin{equation*}
\begin{aligned}
\left|\intr |u|^2\Delta\vphi_{R,0}\,dx\right|
&\le\left(\intr|x|^l |u|^{p+1}\,dx\right)^{\frac{2}{p+1}}
\left(\int_{B_{2R}(0)\setminus B_R(0)}|x|^{-\frac{2l}{p-1}}
|\Delta\vphi_{R,0}|^{\frac{p+1}{p-1}}\,dx\right)^{\frac{p-1}{p+1}} \\
&\le CR^{N-{\frac{2(p+l+1)}{p-1}}}\rightarrow 0 \qquad \textrm{as}\ \ R\to+\infty
\end{aligned}
\end{equation*}
due to the assumption on $p$.
Then letting $R\to+\infty$ in (\ref{thm2.2-1}), we get
\begin{equation}\label{thm2.2-2}
\intr|\nabla u|^2-\mu|x|^{-2}|u|^2\,dx=\intr|x|^l |u|^{p+1}\,dx.
\end{equation}
Finally, combining (\ref{pohozaev id}) and (\ref{thm2.2-2}) we obtain
$$\left(\dfrac{N-2}{2}-\dfrac{l+N}{p+1}\right)\intr|x|^l |u|^{p+1}\,dx=0,$$
but $$\dfrac{N-2}{2}-\dfrac{l+N}{p+1}<0
\qquad \textrm{for} \ \ 1<p<\frac{N+2+2l}{N-2}.$$
Therefore we conclude that $u\equiv0$.

\subsection{Proof of Theorem \ref{thm-morse2}}

Suppose that a weak solution $u$ of (\ref{H}) has finite Morse index in $\RN\setminus\{0\}$.
Then we shall verify that $u$ is identically zero.
For the proof of Theorem \ref{thm-morse2},
the next lemmas which show the behavior of $u$ and $\nabla u$ near the origin and infinity
are needed.

\begin{lem}\label{lem3}
Suppose that a weak solution $u$ of \eqref{H} in $\RN\setminus\{0\}$
is stable in $B_{\ep_0}(0)\setminus\{0\}$.
Then for any $(N+2+2l)/(N-2)\leq p < p_c(l_-,\mu)$,
\begin{equation}\label{behavior zero}
\lim_{|x|\to 0}|x|^{\frac{l+2}{p-1}}|u(x)|=0
\end{equation}
and
\begin{equation}\label{grad behavior zero}
\lim_{|x|\to 0}|x|^{1+\frac{l+2}{p-1}}|\nabla u(x)|=0.
\end{equation}
\end{lem}

\begin{proof} We divide the proof into four steps.
\\
\noindent\emph{Step 1.} We consider the function
$$H(p,\gam,l):=N(p-1)-(\gam+1)l-2(p+\gam).$$
Recalling the proof of Theorem \ref{thm-stable} we have
$$H(p,1,l)\geq0 \qquad \textrm{for}\ \ p\geq\dfrac{N+2+2l}{N-2}$$
and
$$H(p,\gam_M(p,\mu),l)<0 \qquad \textrm{for}\ \ 1<p<p_c(l,\mu).$$
Then there is a $\gam_{\ast}\in[1,\gam_M(p,\mu))$ such that
\begin{equation}\label{gam star}
H(p,\gam_{\ast},l)=0 \qquad \textrm{for}\ \ \dfrac{N+2+2l}{N-2}\leq p<p_c(l,\mu).
\end{equation}
From (a) in Lemma \ref{lem1}, there is a $\ep_{\ast}\in(0,\ep_0)$
such that for every $\ep\in(0,\ep_{\ast}/2)$,
$$\int_{\{\ep<|x|<\ep_{\ast}\}} |x|^l |u|^{p+\gam_{\ast}}\,dx
\le C_1+C_2\,\ep^{N-\frac{(\gam_{\ast}+1)l+2(p+\gam_{\ast})}{p-1}}=C_1+C_2,$$
and thus
$$\int_{B_{\ep_{\ast}}(0)}|x|^l |u|^{p+\gam_{\ast}}\,dx<+\infty.$$
This implies that for any $\eta>0$,
there exists a $\delta=\delta(p,N,\eta,u,l)<\ep_{\ast}$ such that
\begin{equation}\label{int near zero}
\int_{\{|x|<\delta\}}|x|^l |u|^{p+\gam_{\ast}}\,dx<\eta
\end{equation}
for $(N+2+2l)/(N-2)\leq p < p_c(l,\mu)$.
\\
\emph{Step 2. For any $(N+2+2l)/(N-2)\leq p < p_c(l_-,\mu)$,
there exists a small $\kappa_0=\kappa_0(p,N,l)>0$ such that
for every $\kappa\in(0,\kappa_0]$ and open ball $B_{2\rho}(y)$
with $0<|y|<2\ep_0/3$ and $\rho=|y|/8$,
\begin{equation}\label{lemma4.3-step1}
\big{\|}|x|^{-2}\big{\|}_{L^{\frac{N}{2-\kappa}}(B_{2\rho}(y))}
+\big{\|}|x|^l |u|^{p-1}\big{\|}_{L^{\frac{N}{2-\kappa}}(B_{2\rho}(y))} \leq C\,\rho^{-\kappa},
\end{equation}
where $C$ is a positive constant, independent of $y$ and $\rho$.}

For any $y\in\RN$ with $0<|y|<2\ep_0/3$ and $|y|=8\rho$,
a direct calculation gives that
there exists a positive constant $C_1$ such that
\begin{equation}\label{thm2.3-step1-1}
\int_{B_{2\rho}(y)}\left(|x|^{-2}\right)^{\frac{N}{2-\kappa}}\,dx
\leq C_1 \rho^{N-\frac{2N}{2-\kappa}}
\end{equation}
for every $\kappa>0$ sufficiently small.

We turn to the $L^{\frac{N}{2-\kappa}}$-boundedness of $|x|^l|u|^{p-1}$,
and the cases $-2<l<0$ and $l\geq 0$ will be treated separately.

\noindent \textbf{Case 1) $-2<l<0$:} Since $H$ is strictly decreasing in $l$,
we deduce from (\ref{gam star}) that for any $l\in(-2,0)$,
$$H(p,\gam_{\ast},0)<H(p,\gam_{\ast},l)=0
\qquad \textrm{for}\ \ \dfrac{N+2+2l}{N-2}\leq p<p_c(l,\mu),$$
or equivalently $$\frac{N(p-1)}{2}<p+\gam_{\ast}.$$
Then there is a $\kappa_0>0$ sufficiently small such that
$$\dfrac{N(p-1)}{2-\kappa}<p+\gam_{\ast}
\qquad \textrm{for every}\ \ \kappa \in (0,\kappa_0],$$
and therefore there exists a $\xi=\xi(\kappa)>1$ such that
$$\dfrac{N(p-1)}{2-\kappa}\xi=p+\gam_{\ast}
\qquad \textrm{for every}\ \ \kappa \in (0,\kappa_0].$$
Applying H\"{o}lder's inequality together with (b) in Lemma \ref{lem1}
we have that for any $(N+2+2l)/(N-2)\leq p<p_c(l,\mu)$,
\begin{eqnarray}\label{p int near zero1}
\int_{B_{2\rho}(y)}(|x|^l |u|^{p-1})^{\frac{N}{2-\kappa}}\,dx
&\leq& \left(\int_{B_{2\rho}(y)}|x|^l |u|^{p+\gam_{\ast}}\,dx\right)^{\frac{1}{\xi}}
\left(\int_{B_{2\rho}(y)}|x|^{\frac{l}{\xi-1}\left(\frac{N\xi}{2-\kappa}-1\right)}
\,dx\right)^{\frac{\xi-1}{\xi}} \nonumber \\
&=& \left(\int_{B_{2\rho}(y)}|x|^l |u|^{p+\gam_{\ast}}\,dx\right)^{\frac{1}{\xi}}
\left(\int_{B_{2\rho}(y)}|x|^{\frac{l}{\xi-1}\left(\frac{\gam_{\ast}+1}{p-1}\right)}
\,dx\right)^{\frac{\xi-1}{\xi}} \nonumber \\
&\leq& C_2\,\rho^{\frac{N}{\xi}-\frac{(\gam_{\ast}+1)l+2(p+\gam_{\ast})}{(p-1)\xi}
+N\left(1-\frac{1}{\xi}\right)+\frac{(\gam_{\ast}+1)l}{(p-1)\xi}}
= C_2\,\rho^{N-\frac{2(p+\gam_{\ast})}{(p-1)\xi}} \nonumber\\
&=& C_2\,\rho^{N-\frac{2N}{2-\kappa}}
\end{eqnarray}
for some positive constant $C_2$ independent of $\rho$.

\noindent \textbf{Case 2) $l\geq 0$:} We observe
$$H(p,\gam_M(p,\mu),0)<0 \qquad \textrm{for}\ \ 1<p<p_c(0,\mu).$$
Then by the continuity of $H$ in $\gam$,
there is a $\gam_0=\gam_0(p,\mu)\in(1,\gam_M(p,\mu))$ such that
$$H(p,\gam_0,0)<0 \qquad \textrm{for}\ \ 1<p<p_c(0,\mu).$$
This implies
$$\dfrac{N(p-1)}{2}<p+\gam_0
\qquad \textrm{for}\ \ \dfrac{N+2+2l}{N-2}\leq p<p_c(0,\mu),$$
thus there is a $\kappa_0>0$ sufficiently small such that
$$\frac{N(p-1)}{2-\kappa}<p+\gam_0 \qquad \textrm{for every}\ \ \kappa \in (0,\kappa_0].$$
Choosing $\xi=\xi(\kappa)>1$ such that
$$\dfrac{N(p-1)}{2-\kappa}\xi=p+\gam_0
\qquad \textrm{for every}\ \ \kappa \in (0,\kappa_0],$$
we obtain from H\"{o}lder's inequality and (b) in Lemma \ref{lem1} that
for any $(N+2+2l)/(N-2)\leq p<p_c(0,\mu)$,
\begin{eqnarray}\label{p int near zero2}
\int_{B_{2\rho}(y)}(|x|^l |u|^{p-1})^{\frac{N}{2-\kappa}}\,dx
&\leq& \left(\int_{B_{2\rho}(y)}|x|^l |u|^{p+\gam_0}\,dx\right)^{\frac{1}{\xi}}
\left(\int_{B_{2\rho}(y)}|x|^{\frac{l}{\xi-1}\left(\frac{N\xi}{2-\kappa}-1\right)}
\,dx\right)^{\frac{\xi-1}{\xi}} \nonumber \\
&\leq& C_2\,\rho^{\frac{N}{\xi}-\frac{(\gam_0+1)l+2(p+\gam_0)}{(p-1)\xi}
+N\left(1-\frac{1}{\xi}\right)+\frac{(\gam_0+1)l}{(p-1)\xi}} \nonumber\\
&=& C_2\,\rho^{N-\frac{2N}{2-\kappa}}.
\end{eqnarray}
Hence we deduce (\ref{lemma4.3-step1}) from (\ref{thm2.3-step1-1})--(\ref{p int near zero2}).
\\
\emph{Step 3. The behavior of $u$ and $\nabla u$ near the origin.}

Assume $(N+2+2l)/(N-2)\leq p<p_c(l_-,\mu)$.
We regard $u=u(x)$ as solution of the linear equation:
\begin{equation}\label{linear eq}
\Delta u+d(x)u=0 \qquad \textrm{in}\ \ B_{2\rho}(y),
\end{equation}
where $d(x)=\mu|x|^{-2}+|x|^l |u|^{p-1}\in L^{\frac{N}{2-\kappa}}(B_{2\rho}(y))$
by (\ref{lemma4.3-step1}) of Step 1.
Then the well-known result \cite[Theorem 1]{S} say that
every solution $u$ of (\ref{linear eq}) satisfies
\begin{equation}\label{norm on ball}
\|u\|_{L^\infty(B_{\rho}(y))}\le C_s \rho^{-\frac{N}{2}}\|u\|_{L^2(B_{2\rho}(y))}
\end{equation}
and
\begin{equation}\label{grad norm on ball}
\|\nabla u\|_{L^2(B_{\rho}(y))}\le C_s\rho^{-1}\|u\|_{L^2(B_{2\rho}(y))},
\end{equation}
where $C_s$ is a positive constant depending on $p$, $N$ and also on
$$\rho^\kappa\|d\|_{L^{\frac{N}{2-\kappa}}(B_{2\rho}(y))}.$$
In particular, by (\ref{lemma4.3-step1}) we have
$$\rho^\kappa\|d\|_{L^{\frac{N}{2-\kappa}}(B_{2\rho}(y))}\leq C$$
for some positive constant $C$ independent of $\rho$ and $y$,
and therefore the constant $C_s$ in (\ref{norm on ball}) and (\ref{grad norm on ball})
is independent of both $y$ and $\rho$.

On the other hand, we recall that for any $\eta>0$,
there exists a $\delta>0$ satisfying (\ref{int near zero})
for $(N+2+2l)/(N-2)\leq p<p_c(l_-,\mu)$;
here we use the fact that $p=p_c(l,\mu)$ is strictly increasing in $l\geq0$.
Choose any $y\in\RN$ with $0<|y|<\delta/10$ and $\rho=|y|/8$.
Using (\ref{norm on ball}) and H\"{o}lder's inequality
together with $\gam_{\ast}$ satisfying (\ref{gam star}), it follows that
\begin{equation*}
\begin{aligned}
\|u\|_{L^\infty(B_{\rho}(y))}
&\le C_s\rho^{-\frac{N}{2}}\|u\|_{L^2(B_{2\rho}(y))} \\
&\le C_s \rho^{-\frac{N}{2}}\left(\int_{B_{2\rho}(y)}|x|^l |u|^{p+\gam_{\ast}}\,dx\right)^{\frac{1}{p+\gam_{\ast}}}
\left(\int_{B_{2\rho}(y)}|x|^{-\frac{2l}{p+\gam_{\ast}-2}}\,dx\right)^{\frac{p+\gam_{\ast}-2}{2(p+\gam_{\ast})}} \\
&\le C_1 \rho^{-\frac{l+N}{p+\gam_{\ast}}}
\left(\int_{B_{2\rho}(y)}|x|^l |u|^{p+\gam_{\ast}}\,dx\right)^{\frac{1}{p+\gam_{\ast}}}
\end{aligned}
\end{equation*}
for some positive constant $C_1$ independent of $y$.
We observe that $\gam_{\ast}$ satisfies $$\dfrac{l+N}{p+\gam_{\ast}}=\dfrac{l+2}{p-1}$$
from the fact that $H(p,\gam_{\ast},l)=0$.
Then recalling that $\rho=|y|/8$ and (\ref{int near zero}) holds true, we have
$$|u(y)|\le\|u\|_{L^\infty(B_{\rho}(y))}
\le C_2 |y|^{-\frac{l+2}{p-1}}
\left(\int_{B_{2\rho}(y)}|x|^l |u|^{p+\gam_{\ast}}\,dx\right)^{\frac{1}{p+\gam_{\ast}}}
\le C_3|y|^{-\frac{l+2}{p-1}}\eta$$
for some positive constant $C_3$.
In other words, for any $\eta>0$, there exists a $m=\delta/10$ such that
\begin{equation}\label{behavior zero1}
|y|^{\frac{l+2}{p-1}}|u(y)|\leq C_3\,\eta \qquad \textrm{if}\ \ 0<|y|<m,
\end{equation}
which implies (\ref{behavior zero}).

Finally, by \eqref{behavior zero1}, we see that
$$-\Delta u(y)=\mu|x|^{-2}u+|x|^l|u|^{p-1}u=o\left(|y|^{-\frac{l+2}{p-1}-2}\right) \qquad\textrm{as}\ \ |y|\to0.$$
The scaling argument and standard elliptic theory imply
$$|\nabla u(y)|=o\left(|y|^{-\frac{l+2}{p-1}-1}\right) \qquad\textrm{as}\ \ |y|\to0.$$

We complete the proof of Lemma \ref{lem3}.
\end{proof}

\begin{lem}\label{lem4}
Suppose that a weak solution $u$ of \eqref{H} in $\RN\setminus\{0\}$
is stable in $\RN\setminus B_{R_0}(0)$.
Then for any $(N+2+2l)/(N-2)\leq p < p_c(l_-,\mu)$,
\begin{equation}\label{behavior infty}
\lim_{|x|\to +\infty}|x|^{\frac{l+2}{p-1}}|u(x)|=0
\end{equation}
and
\begin{equation}\label{grad behavior infty}
\lim_{|x|\to +\infty}|x|^{1+\frac{l+2}{p-1}}|\nabla u(x)|=0.
\end{equation}
\end{lem}

\begin{proof}
We use Kelvin transformation as in \cite{WY, WY0} by letting
$$v(x)=|x|^{2-N}u\left(\frac{x}{|x|^2}\right)\qquad\textrm{for}\ \ |x|>0\ \ \textrm{small}.$$
Then $v$ satisfies
$$\Delta v+|x|^{-2}v+|x|^m|v|^{p-1}v=0 \qquad\textrm{in}\ \ B_r(0)$$
for small $r>0$, where $m:=(N-2)(p-1)-(l+4)>-2$ since $p>(N+l)/(N-2)$.
Hence we can see that $v$ is a stable weak solution in $B_r(0)\setminus\{0\}$
because $u$ is stable near infinity (see \cite[Proposition 1.6]{DG1} for details).
Therefore we get that
$$|u(x)|=|x|^{2-N}\left|v\left(\frac{x}{|x|^2}\right)\right|=o\left(|x|^{-\frac{l+2}{p-1}}\right)
\qquad\textrm{as}\ \ |x|\to+\infty$$
and $$|\nabla u(x)|=o\left(|x|^{-1-\frac{l+2}{p-1}}\right)
\qquad\textrm{as}\ \ |x|\to+\infty.$$
\end{proof}

When $p=\frac{N+2+2l}{N-2}$, we obtain (\ref{thm2.3-limit1}) and (\ref{thm2.3-limit2})
from Lemma \ref{lem3} and Lemma \ref{lem4}.

Now suppose that $\frac{N+2+2l}{N-2}<p<p_c(l_-,\mu)$.
By the change of variable used in \cite{GS}, we let
$$u(r,\sigma)=r^{-\frac{l+2}{p-1}}v(t,\sigma),\qquad t=\ln{r}.$$
Then $v$ is a solution of the following equation:
$$v_{tt}+A v_t+\Delta_{S^{N-1}}v+Bv+|v|^{p-1}v=0 \qquad \textrm{in}\ \ \mathbb{R}\times S^{N-1}$$
with $$A=N-2-2\,\frac{l+2}{p-1} \qquad \textrm{and} \qquad
B=-\frac{l+2}{p-1}\left(N-2-\frac{l+2}{p-1}\right) +\mu,$$
where $\Delta_{S^{N-1}}$ denotes the Laplace-Beltrami operator on $S^{N-1}$.

Setting
$$E(w):=\int_{S^{N-1}}\left(\frac{1}{2}|\nabla_{S^{N-1}}w|^2
-\frac{B}{2}|w|^2-\frac{1}{p+1}|w|^{p+1}\right)\,d\sigma,$$
we see that
\begin{equation}\label{id1}
A\int_{S^{N-1}}{v_t}^2\,d\sigma
=\frac{d}{dt}\left(E(v)(t)-\frac{1}{2}\int_{S^{N-1}}{v_t}^2\,d\sigma\right).
\end{equation}
Integrating both side of (\ref{id1}) with respect to $t$,
we have that for all $s>0$,
\begin{equation}
\begin{aligned}\label{id2}
A&\int^{s}_{-s}\int_{S^{N-1}}{v_t}^2\,d\sigma\,dt \\
&=E(v)(s)-E(v)(-s)-\frac{1}{2}\int_{S^{N-1}}{v_t}^2(s,\sigma)\,d\sigma
+\frac{1}{2}\int_{S^{N-1}}{v_t}^2(-s,\sigma)\,d\sigma.
\end{aligned}
\end{equation}
The decay estimates in Lemma \ref{lem3} and Lemma \ref{lem4} yield that
\begin{equation}\label{decay est}
\lim_{t\to \pm\infty}v(t,\sigma)=0
\qquad \textrm{and} \qquad
\lim_{t\to \pm\infty}|v_t(t,\sigma)|=\lim_{t\to \pm\infty}|\nabla_{S^{N-1}}v(t,\sigma)|=0,
\end{equation}
where the limits are uniform with respect to $\sigma\in S^{N-1}$.
Then we obtain from (\ref{decay est}) that
$$A\int^{+\infty}_{-\infty}\int_{S^{N-1}}{v_t}^2\,d\sigma\,dt=0,$$
which implies that  $v=v(\sigma)$ (since $A\neq0$).
Using the fact that $v(t,\sigma)\to 0$ as $t\to+\infty$,
we conclude that $v\equiv0$,
and therefore $u\equiv0$.


\section{Existence of a stable solution}
\setcounter{equation}{0}

This section is devoted to the existence of a nontrivial stable weak solution of (\ref{H}) in $\RN$,
that is, the proof of Theorem \ref{thm-exist}.
 {
In \cite{Bae}, the author shows that when $(\mu,p)\in S$, the following equation
$$u''+(N-1)r^{-1}u'+\mu r^{-2}u+r^l |u|^{p-1}u=0,
\qquad \lim_{r\to 0^+}r^{\sqrt{\hmu}-\sqrt{\hmu-\mu}}\,u(r)=\lambda>0$$
has one parameter family of stable regular solutions (see Theorems 3.4 and 3.8 in \cite{Bae}).
The ideas in \cite{Bae0,Bae} are to make use of suitable transformations. The proof for the stability is brief also.
For readers' convenience, we provide a direct proof of Theorem \ref{thm-exist} and moreover explain the stability in detail.}

In view of Theorem \ref{thm-stable}, we need only to consider the case $p\geq p_c(l,\mu)$.
We shall derive some results relevant to positive radial solutions of (\ref{H}).
These results immediately lead to the conclusion of Theorem \ref{thm-exist}.
Hence we write the equation (\ref{H}) in the following radial version:
\begin{equation}\label{radial eq}
\left\{ \begin{array}{ll}
u''+(N-1)r^{-1}u'+\mu r^{-2}u+r^l |u|^{p-1}u=0,\\
\lim_{r\to0^+} r^{\frac{l+2}{p-1}}u(r)=0,\\
\lim_{r\to0^+}(r^{\frac{l+2}{p-1}}u(r))'=0,
\end{array}\right.
\end{equation}
where $r=|x|$.
We first show that a solution of (\ref{radial eq}) in $(0,+\infty)$ exists.

\begin{prop}\label{prop2}
If $(\mu,p)\in S$, then the equation \eqref{radial eq} possesses a positive radial solution in $(0,+\infty)$.
\end{prop}

For the proof of Proposition \ref{prop2}, as well as for other results,
we use the phase-plane analysis and need the following lemma.
For sake of convenience, we use the following notations:
\begin{equation}\label{notation}
A:=N-2-2\,\dfrac{l+2}{p-1},\qquad L^{p-1}:=\dfrac{l+2}{p-1}\left(N-2-\dfrac{l+2}{p-1}\right).
\end{equation}

An elementary calculation gives the following lemma.
\begin{lem}\label{lem-set}
We have $$S=\Sigma,$$
where
$$\Sigma:=\Big{\{}(\mu,p)\in(-\infty,\hmu)\times(1,+\infty)\ :\
p>\dfrac{N+2+2l}{N-2}\ \ \textrm{and}\ \ L^{p-1}>\mu\geq L^{p-1}-\dfrac{A^2}{4(p-1)} \Big{\}}.$$
\end{lem}

\begin{proof}
\noindent \emph{Step 1.} A simple calculation gives the following equivalent result:
\begin{equation}\label{step1}
p>\dfrac{N+2+2l}{N-2} \qquad \textrm{and} \qquad L^{p-1}>\mu
\end{equation}
if and only if
\begin{align*}
\left\{ \begin{array}{ll}
\frac{N+2+2l}{N-2}<p<\frac{l+2}{\nu_-}+1 \qquad &\textrm{if}\ \ 0<\mu<\hmu,\\
p>\frac{N+2+2l}{N-2} \qquad &\textrm{if}\ \ \mu\leq0.
\end{array}\right.
\end{align*}

\noindent \emph{Step 2. We claim that
$$p>\dfrac{N+2+2l}{N-2} \qquad \textrm{and} \qquad \mu\geq L^{p-1}-\dfrac{A^2}{4(p-1)}$$
if and only if
\begin{align*}
\left\{ \begin{array}{ll}
p\geq p_c(l,\mu) \qquad &\textrm{if}\ \ 0<\mu<\hmu\ \ \textrm{and}\ \ N\geq3,\\
p_-\leq p\leq p_+ \qquad &\textrm{if}\ \ \mu_{\ast}\leq\mu\leq0\ \ \textrm{and}\ \ N>10+4l.
\end{array}\right.
\end{align*}}

Let $$\Lambda:=\Big{\{}(\mu,p)\in(-\infty,\hmu)\times\Big{(}\frac{N+2+2l}{N-2},\infty\Big{)}\ :\ \mu\geq L^{p-1}-\dfrac{A^2}{4(p-1)}\Big{\}}.$$
Setting $m:=(l+2)/(p-1)$ and
defining a function $h$ on $(0,\sqrt{\hmu})\times(-\infty,\hmu)$ by
$$h_{\mu}(m)=h(m,\mu):=4m^3+4(l+4-N)m^2+(N-2)(N-10-4l)m+4\mu(l+2),$$
we observe that
$$\Lambda=\left\{m\in(0,\sqrt{\hmu})\,:\, h_{\mu}(m)\geq0\right\}.$$

Now fix $\mu\in(0,\hmu)$. Then there exists a unique $m_0\in(0,\sqrt{\hmu})$
with $m_0=(l+2)/(p_0-1)$ such that $h_{\mu}(m_0)=0$
because $h_{\mu}(0)=4\mu(l+2)>0$ and $h_{\mu}(\sqrt{\hmu})<0$.
This gives that $$\Lambda\cap\left\{(\mu,p)\,:\, 0<\mu<\hmu\right\}=\left\{m\,:\, 0<m\leq m_0\right\}.$$
Let $$m_c:=\dfrac{l+2}{p_c(l,\mu)-1}.$$
By recalling (\ref{gam max}), (\ref{fn f}) and $N=f(p_c(l,\mu),\mu)$,
we see that $h_{\mu}(m_c)=0$, which implies $m_0=m_c$
because $0<m_c<\sqrt{\hmu}$ and $m_0$ is unique on $(0,\sqrt{\hmu})$.
Hence $p_0=p_c(l,\mu)$, and thus
$$\Lambda\cap\left\{(\mu,p)\,:\, 0<\mu<\hmu\right\}=\left\{(\mu,p)\,:\,0<\mu<\hmu,\ p\geq p_c(l,\mu)\right\}.$$

On the other hand, we fix $\mu\leq0$ and two cases $N\leq10+4l$ and $N>10+4l$ are treated separately.
We note that $h_{\mu}$ has extremal values at $m=(N-10-4l)/6$, $\sqrt{\hmu}$.

When $N\leq10+4l$, we have
$h_{\mu}(m)<0$ for all $m\in(0,\sqrt{\hmu})$
since $(N-10-4l)/6\leq0$ and $h_{\mu}(0)=4\mu(l+2)\leq0$.

Consider the case $N>10+4l$, then by using the fact that $h(m,\mu)$ is strictly increasing in $\mu$,
we see that there exists a $\mu_{\ast}<0$ such that
$h_{\mu_{\ast}}(m)\leq0$ for all $m\in(0,\sqrt{\hmu})$ and
$$h_{\mu_{\ast}}\left(\dfrac{N-10-4l}{6}\right)=0
\qquad \textrm{with} \qquad m_{\ast}:=\frac{l+2}{p_{\ast}-1}=\dfrac{N-10-4l}{6},$$
which gives
$$\mu_{\ast}=-\dfrac{(N-10-4l)^2(2N+l-2)}{108(l+2)}
\qquad \textrm{and} \qquad p_{\ast}=\frac{N+2+2l}{N-10-4l}.$$
Then for any $\mu\in[\mu_{\ast},0]$,
there are $m_-=m_-(\mu)$ and $m_+=m_+(\mu)$
in $(0,\sqrt{\hmu})$
such that $m_+\leq m_-$ and $h_{\mu}(m_-)=h_{\mu}(m_+)=0$,
here we denote $m_-:=(l+2)/(p_--1)$ and $m_+:=(l+2)/(p_+-1)$.
Therefore we have that
\begin{equation*}
\begin{aligned}
\Lambda\cap\{(\mu,p)\,:\, \mu\leq0\}&=\{m\,:\,m_+\leq m\leq m_-\}\cap\{\mu\,:\,\mu_{\ast}\leq\mu\leq0\}\\
&=\{(\mu,p)\,:\,\mu_{\ast}\leq\mu\leq0,\ p_-\leq p\leq p_+\},
\end{aligned}
\end{equation*}
so that Step 2 is proved.

In addition, to figure out the graph of two functions $p=p_-(\mu)$, $p_+(\mu)$ in $(\mu,p)$-plane,
we differentiate both side of $h(m_-(\mu),\mu)=0$ with respect to $\mu\in[\mu_{\ast},0]$.
Then it follows
$$\dfrac{\partial h}{\partial \mu}+\dfrac{\partial h}{\partial m}\,
\dfrac{\partial m_-(\mu)}{\partial\mu}\Bigg{|}_{m=m_-(\mu)}=0,$$
and thus we see that $m_-(\mu)$ is strictly increasing in $\mu$
because $\partial h/\partial\mu>0$ and $\partial h/\partial m<0$.
By the same argument, we deduce that $m_+(\mu)$ is strictly decreasing in $\mu$.
These give that $p_-$ is strictly decreasing in $\mu$ and
$p_+$ is strictly increasing in $\mu$.
In particular, when $\mu=0$, we see that
$h_0(m)\geq0$ for all
$$m_+(0)=0<m\leq m_-(0)=\frac{l+2}{p_--1}=\frac{N-l-4-\sqrt{(l+2)(l+2N-2)}}{2},$$
which implies that $p_+(0)=+\infty$ and $p_-(0)=p_c(l,0)$.

By Step 1 and Step 2, we complete the proof of Lemma \ref{lem-set}.
\end{proof}

\begin{proof}[\textbf{Proof of Proposition \ref{prop2}}]
Consider the following change of variable:
\begin{equation}\label{change var}
w(t)=r^{\frac{l+2}{p-1}}u(r),\qquad t=\ln{r},
\end{equation}
then (\ref{radial eq}) is transformed into the following autonomous ODE:
\begin{equation*}
\left\{ \begin{array}{ll}
w''(t)+Aw'(t)-(L^{p-1}-\mu)w(t)+|w(t)|^{p-1}w(t)=0,\qquad -\infty<t<\infty,\\
\lim_{t\to-\infty} w(t)=0,\\
\lim_{t\to-\infty}w'(t)=0.
\end{array}\right.
\end{equation*}
Setting $w'(t)=v(t)$, we have the following system:
\begin{equation}\label{w-v eq}
\left \{ \begin{aligned}
&\ w'=v,\\
&\ v'=-Av+(L^{p-1}-\mu)w-|w|^{p-1}w.
\end{aligned}\right.
\end{equation}
Then the equilibrium solutions of (\ref{w-v eq})
are $(0,0)$ and $(\pm\,w_0,0)$ with $$w_0=(L^{p-1}-\mu)^{\frac{1}{p-1}};$$
here the latter exists if
\begin{equation}\label{mu cond1}
L^{p-1}>\mu.
\end{equation}
As $(\mu,p)\in S=\Sigma$, direct calculations show that
$(0,0)$ is a saddle point and $(w_0,0)$ is asymptotically stable.

Now, at $(0,0)$, we see that the system (\ref{w-v eq}) has an unstable curve $W^u$
leaving $(0,0)$, where $W^u$ is tangent to the line spanned by the eigenvalue $\lambda_0^+$ at $(0,0)$.
Let the corresponding solution be $$W^u(t)=(w(t), v(t))$$
and $J=(-\infty,\tau)$ be the maximal interval where $W^u(t)$ is defined.
We claim that $w$ does not change sign.
Indeed, suppose that there exists a $t_0\in J$ such that $w(t_0)=0$.
If we define
$$E_w(t):=\frac{1}{2}(w')^2-\dfrac{1}{2}\left(L^{p-1}-\mu\right)w^2+\frac{1}{p+1}|w|^{p+1},$$
then we see that $E_w(t_0)\ge0$ and
$$\dfrac{d}{dt}E_w(t)=-A(w')^2\leq0$$
because $A>0$.
Moreover, the initial conditions of \eqref{radial eq} give
$$E_w(t)\rightarrow 0 \qquad \textrm{as}\ \ t \rightarrow -\infty.$$
Hence we obtain
$$E_w(t)\equiv0 \qquad \textrm{for}\ \ t\leq t_0.$$
In particular, we have $w'(t_0)=0$ because $E_w(t_0)=0$.
Now using the uniqueness of solution for the initial value problem,
it follows that $w\equiv0$; it is a contradiction to our assumption.
Therefore, either $w>0$ or $w<0$.

We can now assume that $w(t)>0$ for all $t\in J$.
Finally we claim that $\tau=+\infty$.
In fact, by Theorem in \cite[Chapter 8.5]{HS},
it suffices to show that
there is a compact set $K$ in $\{(w,v)\,:\,w>0\}$ such that
$(w(t),v(t))\in K$ for all $t\in J$.
Using $E_w\le0$, we have $$0\le v^2\le \left(L^{p-1}-\mu\right)w^2-\frac{2}{p+1}|w|^{p+1}.$$
As $p>1$, this implies readily that $|w|$ is bounded, so is $v$.
Then it concludes $\tau=+\infty$, and thus $W^u(t)$ is defined in $(-\infty,+\infty)$.
Therefore, we get a positive radial solution of (\ref{radial eq}) in $(0,+\infty)$.
\end{proof}

Furthermore, we find a property of positive radial solutions of (\ref{radial eq}) in $(0,+\infty)$
in the following proposition:

\begin{prop}\label{prop3}
Let $u$ be any positive radial solution of \eqref{radial eq} in $(0,+\infty)$ satisfying
\begin{equation}\label{radial-condition}
\lim_{r\to 0^+}r^{\sqrt{\hmu}-\sqrt{\hmu-\mu}}\,u(r)=\lambda>0.
\end{equation}
For any $(\mu,p)\in S$, it follows that
\begin{enumerate}[{\em (a)}]
\item \label{radial-envelope}
$u(r) < U_s(r):=(L^{p-1}-\mu)^{\frac{1}{p-1}}\,r^{-\frac{l+2}{p-1}}\qquad
\textrm{for}\ \ r>0$;
\item \label{radial-stable} $u$ is a stable weak solution in $W^{1,2}_{loc}(\RN)$.
\end{enumerate}
\end{prop}

\begin{rem}\em
(a) Bidaut-V\'{e}ron and V\'{e}ron (see Theorem 3.2 and Remark 3.2 in \cite{BV})
obtained that if a positive solution $u$ of (\ref{H}) in $B_R(0)\setminus\{0\}$ satisfies
$$|x|^{\frac{l+2}{p-1}}u(x)\in L^{\infty}_{loc}(B_R(0)\setminus\{0\})$$
with $p\in(1,+\infty)\setminus\{(N+2+2l)/(N-2)\}$ and $\mu<L^{p-1}$,
then it holds that either $u$ is a weak solution of (\ref{H}) in $B_R(0)$ satisfying
$$\lim_{|x|\to 0}|x|^{\sqrt{\hmu}-\sqrt{\hmu-\mu}}u(x)=\lambda>0$$
or else $$\lim_{|x|\to 0}|x|^{\frac{l+2}{p-1}}u(x)=\omega(\sigma)
\qquad \textrm{uniformly in}\ \ \sigma=\frac{x}{|x|},$$
where $\omega$ is a positive solution of
$$\Delta_{S^{N-1}}\omega-\left\{\dfrac{l+2}{p-1}\left(N-2-\dfrac{l+2}{p-1}\right)-\mu\right\}\omega+\omega^p=0
\qquad \textrm{on}\ \ S^{N-1}.$$
From this result, the condition (\ref{radial-condition}) is quite natural.
Moreover any positive radial solution satisfying (\ref{radial-condition})
becomes weak solution of (\ref{H}) in $W^{1,2}_{loc}(\RN)$. \\
(b) By a direct calculation, we easily see that $U_s$ is a singular positive
radial solution of (\ref{H}) in $\RN\setminus\{0\}$ if $\mu<L^{p-1}$.
Moreover it is stable in $\RN\setminus\{0\}$
when $(\mu,p)\in S$.
\end{rem}

\begin{proof}[\textbf{Proof of Proposition \ref{prop3}}]
Let $u$ be a positive radial solution of (\ref{radial eq})
in $(0,+\infty)$ satisfying (\ref{radial-condition}).
Using (\ref{change var}),
(\ref{radial eq}) is transformed into the following autonomous equation:
\begin{equation*}
\left\{ \begin{array}{ll}
w''(t)+Aw'(t)-(L^{p-1}-\mu)w(t)+w^p(t)=0,\qquad -\infty<t<\infty,\\
\lim_{t\to-\infty} w(t)=0,\\
\lim_{t\to-\infty}w'(t)=0.
\end{array}\right.
\end{equation*}
We first prove that the graph of $u$
does not intersect that of singular solution $U_s$.
To this end, we apply the argument used in Proposition 3.7 (iv) of \cite{W}
and state the proof here.
We shall verify that $$w(t)<(L^{p-1}-\mu)^{\frac{1}{p-1}}=w_0.$$
Suppose that we can choose $t_0=\min\{\,t\,:\,w(t)=w_0\}$.
Then we obtain that
\begin{equation}\label{exp w}
\dfrac{d}{dt}(e^{At}w')>0 \qquad \textrm{on}\ \ (-\infty, t_0)
\end{equation}
because $w''+Aw'=(L^{p-1}-\mu-w^{p-1})w>0$ on $(-\infty, t_0)$.
Now we claim that $w'>0$ on $(-\infty,t_0)$.
Indeed, we can choose an interval $I\subseteq(-\infty,t_0)$
such that $w'>0$ on $I$ (because $0=w(-\infty)<w(t_0)=w_0$).
If $I\subsetneq(-\infty,t_0)$,
then there exists a $t_{\ast}\in(-\infty,t_0)\setminus I$ such that $w'(t_{\ast})=0$.
Thus we deduce from (\ref{exp w}) that for all $t<t_{\ast}$,
$$\int^{t_{\ast}}_t\dfrac{d}{ds}\left(e^{At}w'(s)\right)\,ds
=-e^{At}w'(t)>0,$$
which implies $w'(t)<0$ for all $t\in(-\infty,t_{\ast})$.
This gives that $0=w(-\infty)>w(t)$ for all $t\in(-\infty,t_{\ast})$, a contradiction.
Therefore we conclude that $I=(-\infty,t_0)$, that is, $w'(t)>0$ on $(-\infty,t_0)$ as claimed.

Now setting $w'(t)=v(w)$, $v$ satisfies
$$\dfrac{dv}{dw}=-A+\dfrac{{w_0}^{p-1}w-w^p}{v}
\qquad \textrm{on}\ \ (0,w_0]$$
with $v(w)>0$ in $(0,w_0)$ and $v(w)\to 0$ as $w\to 0^+$.
We observe that the graph of $v=v(w)$ always intersects
all lines $v=c(w_0-w)$ with $c>0$.
For each $c>0$, let $w_c$ be the smallest point of the intersection.
Then we have
\begin{eqnarray*}
-c\ \ \leq\ \ \dfrac{dv}{dw}(w_c)
&=&-A+\dfrac{{w_0}^{p-1}w_c-{w_c}^p}{c(w_0-w_c)} \\
&=&-A+\dfrac{p{\overline{w}_c}^{p-1}-{w_0}^{p-1}}{c}
\qquad \textrm{for some}\ \ \overline{w}_c\in(w_c,w_0) \\
&<&-A+\dfrac{(p-1){w_0}^{p-1}}{c}.
\end{eqnarray*}
Hence we see that
$$c^2-Ac+(p-1){w_0}^{p-1}>0 \qquad \textrm{for all}\ \ c>0,$$
which yields $A^2-4(p-1)(L^{p-1}-\mu)<0$, a contradiction.
Thus we reach our conclusion (\ref{radial-envelope}).

Finally, we use (\ref{radial-envelope}) to prove the stability of $u$.
By a straightforward calculation, we see that
$\hmu\geq \mu+p(L^{p-1}-\mu)$ is equivalent to $$\mu\geq L^{p-1}-\dfrac{A^2}{4(p-1)}.$$
Then for any $\phi\in C^1_c(\RN)$,
\begin{equation*}
\begin{aligned}\label{singular sol-stable}
Q_u(\phi)&=\intr |\nabla\phi|^2-\left(\mu|x|^{-2}+p|x|^l|u|^{p-1}\right)\phi^2\,dx \\
&>\intr |\nabla\phi|^2-\left(\mu|x|^{-2}+p|x|^l|U_s|^{p-1}\right)\phi^2\,dx \\
&=\intr |\nabla\phi|^2-\Big{(}\mu+p(L^{p-1}-\mu)\Big{)}|x|^{-2}\phi^2\,dx \\
&\geq\intr |\nabla\phi|^2-\hmu\,|x|^{-2}\phi^2\,dx \geq 0,
\end{aligned}
\end{equation*}
where the latter inequality is obtained from Hardy's inequality.
Therefore $u$ is stable.
\end{proof}

\begin{proof}[\textbf{Proof of Theorem \ref{thm-exist}}]
When $(\mu, p)\in S$, by Propositions \ref{prop2} and \ref{prop3},
there exists a positive radial solution $u_{\lambda}$ of (\ref{radial eq}) in $(0,+\infty)$
satisfying $$\lim_{r\to 0^+}r^{\sqrt{\hmu}-\sqrt{\hmu-\mu}}u_{\lambda}(r)=\lambda>0$$
and the solution $u_{\lambda}$ is stable.
Therefore, we conclude that (\ref{H}) admits a family of stable positive radial solutions
in $W^{1,2}_{loc}(\RN)$ if $(\mu,p)\in S$.
\end{proof}
\ \\
\ \\
{\bf Acknowledgement}\\
W. Jeong was supported by Priority Research Centers
Program (No.2010-0029638) and Mid-career Researcher
Program (No.2010-0014135) through the National Research Foundation
of Korea (NRF) funded by the Ministry of Education, Science and Technology (MEST).
Y. Lee was supported by Center for Advanced Study in Theoretical Science,
National Taiwan University.
 {
The authors would like to thank their advisor J. Byeon for his valuable guide and encouragement.
The authors thank also S. Bae for giving helpful comments
and sending us his preprint \cite{Bae} including results in Section 5 and Figures 1-2.}



\end{document}